\documentclass{amsart}
\usepackage{graphicx}
\usepackage{amsmath}
\usepackage{amsthm}
\usepackage{amsfonts}
\usepackage{amssymb, amscd}

\usepackage{color}

\usepackage[all]{xy}

\newtheorem{thm}{Theorem}[section]
\newtheorem{cor}[thm]{Corollary}
\newtheorem{lem}[thm]{Lemma}
\newtheorem{proposition}[thm]{Proposition}
\newtheorem{example}[thm]{Example}
\theoremstyle{definition}
\newtheorem{df}[thm]{Definition}
\theoremstyle{remark}
\newtheorem{remark}[thm]{Remark}
\numberwithin{equation}{section}

\newcommand{\K}{\mathbb K}

\newcommand{\g}{ \mathfrak g}

\begin{document}

\title[Generalized representations of   $3$-Hom-Lie algebras ]
{ Generalized representations of   $3$-Hom-Lie algebras }%
\author{ S. Mabrouk, A. Makhlouf, S. Massoud }%
\address{Abdenacer Makhlouf, Universit\'{e} de Haute Alsace,  IRIMAS-d\'epartement de Math\'{e}matiques, 
6, rue des Fr\`{e}res Lumi\`{e}re F-68093 Mulhouse, France}%
\email{Abdenacer.Makhlouf@uha.fr  }
\address{ Sami  Mabrouk, University of Gafsa, Faculty of Sciences Gafsa, 2112 Gafsa, Tunisia}%
\email{ \ Mabrouksami00@yahoo.fr}
\address{ Sonia MASSOUD Universit\'{e}, de Sfax,  Facult\'{e} des Sciences, Sfax Tunisia}%
\email{ \  sonia.massoud2015@gmail.com}


\date{}
%
\begin{abstract}The propose of this paper is to extend   generalized representations of  $3$-Lie algebras to Hom-type algebras. We  introduce the concept of generalized representation of  multiplicative $3$-Hom-Lie algebras,  develop the
corresponding cohomology theory and study semi-direct products. We provide a key construction,  various examples  and computation of $2$-cocycles of the new cohomology. Also, we give a connection between
 a split abelian extension of a $3$-Hom-Lie algebra and a generalized semidirect
product $3$-Hom-Lie algebra.

\end{abstract}
\maketitle

\section*{Introduction}
The first instances of ternary Lie algebras appeared first in Nambu's generalization of Hamiltonian mechanics
\cite{N}, which was formulated algebraically  by Takhtajan \cite{T}. The structure of $n$-Lie algebras was
studied by Filippov \cite{Filippov} then completed by Kasymov in \cite{Kasymov}.

The representation theory of
$n$-Lie algebras was first introduced by Kasymov in \cite{Kasymov}. The adjoint representation is defined by
the ternary bracket in which two elements are fixed. Through fundamental objects one may also
represent a $3$-Lie algebra and more generally an $n$-Lie algebra by a Leibniz algebra (\cite{DT}). The cohomology of $n$-Lie algebras, generalizing the Chevalley-Eilenberg Lie algebras cohomology, was  introduced by Takhtajan \cite{Tcohomology} in its simplest form, later a complex adapted
to the study of formal deformations was introduced by Gautheron \cite{Gautheron}, then reformulated
by Daletskii and Takhtajan \cite{DT} using the notion of base Leibniz algebra of an $n$-Lie algebra. In \cite{AMS11,AKMS}, the structure and cohomology of $3$-Lie algebras induced by Lie algebras has been
investigated.

The concept of  generalized representation of a $3$-Lie
algebra was introduced by Liu, Makhlouf and Sheng in \cite{LuiMakhloufSheng}. They study the corresponding   generalized semidirect product $3$-Lie algebra and
cohomology theory.  Furthermore, they describe general abelian extensions
of 3-Lie algebras using Maurer-Cartan elements. Non-abelian extensions was explored in \cite{SongRepHom}.

The aim of this paper is to extend the concept of  generalized representation of  $3$-Lie
algebras to Hom-type algebras. The  notion of Hom-Lie algebras was introduced by Hartwig, Larsson, and Silvestrov
in \cite{HartwigLarssonSilvestrov} as part of a study of deformations of the Witt and the Virasoro algebras. The $n$-Hom-Lie algebras and various generalizations of $n$-ary algebras were considered in \cite{AtMakh}.  In a
Hom-Lie algebra, the Jacobi identity is twisted by a linear map, called the Hom-Jacobi
identity. In particular, representations and cohomologies of Hom-Lie algebras were studied in \cite{ShengRepHom}, while the representations and cohomology of $n$-Hom-Lie algebras were first studied in \cite{AmmarSamiMakhlouf}.

The paper is organized as follows. in Section 1, we provide some basics about $3$-Hom-Lie algebras, representations and cohomology. The second Section includes the new concept of generalized representation of a $3$-Hom-Lie algebra, extending to Hom-type algebras the notion and results obtained in
\cite{LuiMakhloufSheng}. We define a corresponding semi-direct product and provide a twist procedure leading  generalized representations of $3$-Hom-Lie algebras starting from generalized representations of $3$-Hom-Lie algebras and algebra maps. In Section 3, we construct a new cohomology corresponding the generalized representations and show examples. In the last section we discuss abelian extensions of multiplicative  $3$-Hom-Lie algebras.

\section{Representation of $3$-Hom-Lie algebras}
The aim of this section is to recall some basics about $3$-Lie algebras and $3$-Hom-Lie algebras. We refer mainly to  \cite{Filippov} and \cite{AtMakh}.
In this paper, all  vector spaces are considered over a field $\mathbb{K}$ of characteristic $0$.
\begin{df} A $3$-Lie algebra is a pair  $(\mathfrak{g}, [\cdot,\cdot,\cdot])$ consisting of a $\mathbb{K}$-vector space $\mathfrak{g}$
  and a trilinear skew-symmetric multiplication $[\cdot,\cdot,\cdot]$ satisfying the Filippov-Jacobi identity: for  $x,y,z,u,v$ in $\mathfrak{g}$
\begin{equation}\label{NambuIdentity}
[u,v,[x,y,z]]= [[u,v,x],y,z]+[x,[u,v,y],z]
 +  [x,y,[u,v,z]].
  \end{equation}
  \end{df}
  In this paper, we are dealing with $3$-Hom-Lie algebras corresponding to the following definition.
\begin{df} A $3$-Hom-Lie algebra is a triple $(\mathfrak{g}, [\cdot,\cdot,\cdot], \alpha)$ consisting of a $\mathbb{K}$-vector space $\mathfrak{g}$,  a trilinear skew-symmetric multiplication $[\cdot,\cdot,\cdot]$ and
 an algebra map $\alpha:\mathfrak{g}\rightarrow \mathfrak{g}$  satisfying the  Hom-Filippov-Jacobi identity: for  $x,y,z,u,v$ in $\mathfrak{g}$
\begin{equation}[\alpha(u),\alpha(v),[x,y,z]]= [[u,v,x],\alpha(y),\alpha(z)]+[\alpha(x),[u,v,y],\alpha(z)] +  [\alpha(x),\alpha(y),[u,v,z]]. \label{HomNambuIdentity} 
  \end{equation}
  \end{df}
  \begin{remark} There is more general definition of  $3$-Hom-Lie algebras which are given by  a quadruple $(\mathfrak{g}, [\cdot,\cdot,\cdot], \alpha_1,\alpha_2)$ consisting of a $\mathbb{K}$-vector space $\mathfrak{g}$,
  two linear maps $\alpha_1,\alpha_2:\mathfrak{g}\rightarrow \mathfrak{g}$ and a trilinear skew-symmetric multiplication $[\cdot,\cdot,\cdot]$ satisfying the following generalized   Hom-Filippov-Jacobi identity: for  $x,y,z,u,v$ in $\mathfrak{g}$
\begin{equation}[\alpha_1(u),\alpha_2(v),[x,y,z]]= [[u,v,x],\alpha_1(y),\alpha_2(z)]+[\alpha_1(x),[u,v,y],\alpha_2(z)] +  [\alpha_1(x),\alpha_2(y),[u,v,z]].  
  \end{equation}
  We get our class of $3$-Hom-Lie algebras when $\alpha_1=\alpha_2=\alpha$ and where alpha is an algebra map.This kind  of algebras are usually called multiplicative $3$-Hom-Lie algebras.
 \end{remark}

\begin{proposition}
  Let  $(\mathfrak{g}, [\cdot,\cdot,\cdot])$ be a  $3$-Lie algebra and $\alpha:\mathfrak{g}\to\mathfrak{g}$ be a $3$-Lie algebra morphism. Then $(\mathfrak{g}, [\cdot,\cdot,\cdot]_\alpha:= \alpha\circ[\cdot,\cdot,\cdot],\alpha)$ is a  $3$-Hom-Lie algebra.
\end{proposition}

Let  $(\mathfrak{g}, [\cdot,\cdot,\cdot],\alpha)$ be a   $3$-Hom-Lie algebra,
elements in $\wedge^2 \mathfrak{g}$ are called {\bf fundamental objects} of the $3$-Hom-Lie algebra $(\mathfrak{g},[\cdot,\cdot,\cdot],\alpha)$. There is a bilinear operation $[\cdot,\cdot]_{\mathcal{L}}$ on $  \wedge^{2} \mathfrak{g}$, which is given by
\begin{equation}\label{eq:bracketfunda}
[X,Y]_{\mathcal{L}}=[x_1,x_2,y_1]\wedge \alpha(y_2)+\alpha(y_1)\wedge[x_1,x_2,y_2],\quad \forall X=x_1\wedge x_2, ~Y=y_1\wedge y_2,
\end{equation}
and a linear map $\overline{\alpha}$ on $\wedge^2 \mathfrak{g}$ defined by $\overline{\alpha}(X)=\alpha(x_1) \wedge \alpha(x_2)$, for simplicity, we will write $\overline{\alpha}(X)=\alpha(X)$.
It is well-known that $(\wedge^2 \mathfrak{g},[\cdot,\cdot]_{\mathcal{L}},\overline{\alpha})$ is a  Hom-Leibniz algebra (\cite{DT}).

\begin{df}\label{def2.2}
A representation of a   $3$-Hom-Lie algebra  $(\mathfrak{g},[\cdot,\cdot,\cdot],\alpha)$ on a vector space $V$ with respect to $A \in gl(V)$ is a  skew-symmetric linear map $\rho: \wedge^2 \mathfrak{g}\rightarrow End(V)$ such that{\small
\begin{align}
\label{ConditionRep1}\rho(\alpha(x_1),\alpha(x_{2}))\circ A  &= A\circ \rho(x_{1},x_{2}),\\
\label{ConditionRep2}\rho(\alpha(x_{1}),\alpha(x_{2}))\rho(x_{3},x_{4})-\rho(\alpha(x_{3}),\alpha(x_{4}))\rho(x_{1},x_{2}) &=\big(\rho([x_{1},x_{2},x_{3}],\alpha(x_{4})) -\rho([x_{1},x_{2},x_{4}],\alpha(x_{3}))\big)\circ A,\\
\label{ConditionRep3}\rho([x_{1},x_{2},x_{3}],\alpha(x_{4}))\circ A-\rho(\alpha(x_{2}),\alpha(x_{3}))\rho(x_{1},x_{4}) &=
\rho(\alpha(x_{3}),\alpha(x_{1}))\rho(x_{2},x_{4})
+\rho(\alpha(x_{1}),\alpha(x_{2}))\rho(x_{3},x_{4}).
\end{align}}
for $x_1, x_2, x_3$ and $x_4$ in $\mathfrak{g}$.
\end{df}

\begin{thm}
Let $(\mathfrak{g},[\cdot,\cdot,\cdot])$ be a $3$-Lie algebra,  $(V,\rho)$ be a representation,  $\alpha : \mathfrak{g}\rightarrow \mathfrak{g}$ be a $3$-Lie algebra morphism and $A : V\rightarrow V $ be a linear map  such that $A \circ \rho(x_{1},x_{2})= \rho(\alpha(x_{1}),\alpha(x_{2}))\circ A$.
Then $(V, \widetilde{\rho}:=A\circ \rho,A)$ is a representation of the   $3$-Hom-Lie algebra $(\mathfrak{g}, [\cdot,\cdot,\cdot]_{\alpha}:=\alpha\circ[\cdot,\cdot,\cdot] ,\alpha)$.

\end{thm}
\begin{proof}
Let $x_{i}\in \mathfrak{g}$, where $1\leq i\leq5$. Then we have
\begin{align*}
&\widetilde{\rho}([x_{3},x_{4},x_{5}]_{\alpha},\alpha(x_{1}))\circ A-\widetilde{\rho}(\alpha(x_{3}),\alpha(x_{4}))\widetilde{\rho}(x_{5},x_{1})
-\widetilde{\rho}(\alpha(x_{4}),\alpha(x_{5}))\widetilde{\rho}(x_{3},x_{1})-\widetilde{\rho}(\alpha(x_{5}),\alpha(x_{3}))\widetilde{\rho}(x_{4},x_{1})\\
&= A^{2}\circ(\rho([x_{3},x_{4},x_{5}],x_{1})-\rho(x_{3},x_{4})\rho(x_{5},x_{1})-\rho(x_{4},x_{5})\rho(x_{3},x_{1})-\rho(x_{5},x_{3})\rho(x_{4},x_{1}))\\
&=0.
\end{align*}
The second condition \eqref{ConditionRep2} is obtained similarly.

\end{proof}
The previous result allows to twist along morphisms a 3-Lie algebra with a representation to a $3$-Hom-Lie algebra with a corresponding representation.

\begin{proposition}\label{lem:semidirectp}
Let $(\mathfrak{g},[\cdot,\cdot,\cdot],\alpha)$ be a   $3$-Hom-Lie algebra, $V$  be a vector space,  $ A \in gl(V)$ and $\rho:
\wedge^2 \mathfrak{g}\rightarrow gl(V)$   be a skew-symmetric linear
map. Then $(V;\rho,A)$ is a representation of   $3$-Hom-Lie algebra $\mathfrak{g}$ if and only if there
is a   $3$-Hom-Lie algebra structure $(\mathfrak{g}\oplus V,[\cdot,\cdot,\cdot]_{\rho},\alpha_{\mathfrak{g}\oplus V})$
on the direct sum of vector spaces  $\mathfrak{g}\oplus V$, defined by
\begin{equation}\label{eq:sum}
[x_1+v_1,x_2+v_2,x_3+v_3]_{\rho}=[x_1,x_2,x_3]+\rho(x_1,x_2)v_3+\rho(x_3,x_1)v_2+\rho(x_2,x_3)v_1,
\end{equation}

and $\alpha_{\mathfrak{g}\oplus V}=\alpha + A$, for all $x_i\in \mathfrak{g}, v_i\in V, 1\leq i\leq 3$. The obtained $3$-Hom-Lie algebra is denoted by $\mathfrak{g}\ltimes_\rho V$ and  called semidirect product.


\end{proposition}
Let $(\mathfrak{g},[\cdot,\cdot,\cdot],\alpha)$ be a   $3$-Hom-Lie algebra and $(V,\rho,A)$ be a representation of $\mathfrak{g}$. We denote by $C_{\alpha,A}^{p}(\mathfrak{g},V)$ the space of all linear maps $\varphi:\underbrace{\wedge^2 \mathfrak{g}\otimes ...\otimes \wedge^2 \mathfrak{g}} _{(p-1)} \wedge \mathfrak{g} \longrightarrow V$  satisfying:
 $$A\circ\varphi(X_1\otimes...\otimes X_{p-1},x)=\varphi(\alpha(X_{1})\otimes...\otimes \alpha(X_{p-1}),\alpha(y)), \ \forall \ X_1,...,X_{p-1}\in \wedge^2 \mathfrak{g},\ y\in \mathfrak{g}.$$
  Let $\varphi$ be a $(p-1)$-cochain, the coboundary operator $\delta_{\rho} : C_{\alpha,A}^{p-1}(\mathfrak{g},V)\longrightarrow C_{\alpha,A}^{p}(\mathfrak{g},V)$ is given by

\begin{eqnarray}
\nonumber&&(\delta_\rho\varphi)(X_1,\cdots ,X_p,z)\\
\nonumber&&= \sum_{1\leq j<k}(-1)^j\varphi(\alpha(X_1),\cdots ,\hat{X}_j,\cdots ,\alpha(X_{k-1}),[X_j,X_k]_{\mathcal{L}},\alpha(X_{k+1}),\cdots ,\alpha(X_{p}),\alpha(z))\\
\nonumber&&+\sum_{j=1}^p(-1)^j\varphi(\alpha(X_1),\cdots ,\hat{X}_j,\cdots ,\alpha(X_{p}),[X_j,z])\\
\nonumber&&+\sum_{j=1}^p(-1)^{j+1}\rho(\alpha^p(X_j))\varphi(X_1,\cdots ,\hat{X}_j,\cdots ,X_{p},z)\\
\label{eq:drho}&&+(-1)^{p+1}\Big(\rho(\alpha^p(y_{p}),\alpha^p(z))\varphi(X_1,\cdots ,X_{p-1},x_{p} ) +\rho(\alpha^p(z),\alpha^p(x_{p}))\alpha(X_1,\cdots ,X_{p-1},y_{p} ) \Big),
\end{eqnarray}
for all   $X_i=(x_i,y_i)\in\wedge^2 \mathfrak{g}$, $z\in \mathfrak{g}$ and where $[X_i,z]=[x_i,y_i,z]$. An element $\varphi\in {C}_{\alpha, A}^{p-1}(\mathfrak{g},V)$ is called a $p$-cocycle if $\delta_\rho\varphi=0$. It is called a $p$-coboundary if there exists some $f\in {C}_{\alpha,A}^{p-2}(\mathfrak{g},V)$ such that $\varphi=\delta_\rho f$. Denote by $Z_{3HL}^p(\mathfrak{g};V)$ and $B_{3HL}^p(\mathfrak{g};V)$ the sets of $p$-cocycles and $p$-coboundaries respectively. Then  the $p$-th cohomology group is \begin{equation}\label{eq:cohomology}H_{3HL}^p(\mathfrak{g};V)=Z_{3HL}^p(\mathfrak{g};V)/B_{3HL}^p(\mathfrak{g};V).\end{equation}

In \cite{NR bracket of n-Lie}, the authors constructed a graded Lie algebra structure by which one can describe an $n$-Leibniz algebra structure as a canonical structure. Here, we give the precise formulas for the $3$-Hom-Lie algebra case, generalizing the result in \cite{LuiMakhloufSheng}.

  Set  $C_{\alpha,\alpha}(\mathfrak{g},\mathfrak{g})=\oplus_{p\geq0} C_{\alpha,\alpha}^p(\mathfrak{g},\mathfrak{g})$.
  Let $\varphi\in C_{\alpha,\alpha}^q(\mathfrak{g},\mathfrak{g})$, $\psi\in C_{\alpha,\alpha}^p(\mathfrak{g},\mathfrak{g})$,  $p,q\geq 0$,  $X_i=x_i\wedge y_i\in \wedge^2 \mathfrak{g}$ for $i=1,2,\cdots ,p+q$ and $x\in \mathfrak{g}$. For each subset
 $J=\{j_1,\cdots,j_{p}\}_{j_1<\cdots<j_{p}}\subset N\triangleq \{1,2,\cdots,p+q\}$,
  let $I=\{i_1,\cdots,i_q\}_{i_1<\cdots<i_q}=N/J$.

Define on
the graded vector space $C_{\alpha,\alpha}(\mathfrak{g},\mathfrak{g})$  the graded commutator bracket
\begin{equation}\label{defdeGraddBracket}
[\varphi,\psi]^{3HL}=(-1)^{pq}j_{\psi}^{\alpha}(\varphi)-j_{\varphi}^{\alpha}(\psi)=(-1)^{pq}\varphi\circ_{\alpha}\psi-\psi\circ_{\alpha}\varphi,\end{equation}
where $\varphi\circ_\alpha\psi \in C_{\alpha,\alpha}^{p+q}(\mathfrak{g},\mathfrak{g})$ is defined by
 \begin{align*}
& j_{\psi}^{\alpha}(\varphi)(X_{1},...,X_{p+q},x)
 =\varphi\circ_{\alpha}\psi(X_{1},...,X_{p+q},x)\\
 &=\sum_{J,j_{q}<i_{k+1}\leq p+q}(-1)^{(J,I)}
(\varphi(\alpha^p(X_{i_{1}}),...,\alpha^p(X_{i_{k}}),\psi(X_{j_{1}},...,X_{j_{p}},~) \bullet_\alpha X_{i_{k+1}},...,\alpha^p(X_{i_{q}}),\alpha^p(x))\\
&+\sum_{J}(-1)^{(J,I)} (-1)^{q}(\varphi(\alpha^p(X_{i_{1}}),...,\alpha^p(X_{i_{q}}),\psi(X_{j_{1}},...,X_{j_{p}},x)),
\end{align*}
where $$\psi(X_{j_{1}},...,X_{j_{q}},~) \bullet_\alpha X_{i_{k+1}}=\psi(X_{j_{1}},...,X_{j_{p}},x_{i_{k+1}})\wedge\alpha^p(y_{i_{k+1}})
+\alpha^p(x_{i_{k+1}})\wedge\psi(X_{j_{1}},...,X_{j_{p}},y_{i_{k+1}})$$
where $k$ is uniquely determined by the condition $j_{p}\leq i_{k+1}$ and if $j_{p}\leq i_{1}$ then $j_{p}=p,$ $i_{1}=p+1$ and $(-1)^{(J,I)}$ is the sign of the permutation $(J,I)=(j_1,...,j_p,i_1,..., i_q)$ of $N$.\\

We need  the following lemma to establish a structure of graded Lie algebra on $C_{\alpha,\alpha}(\mathfrak{g},\mathfrak{g})$ .
\begin{lem}

We have $j_{[\varphi, \psi]^{3HL}}
= -[j_{\varphi}^{\alpha}, j_{\psi}^{\alpha}]$
for all $\varphi, \psi \inž C_{\alpha, \alpha}(\mathfrak{g},\mathfrak{g})$, where $[\cdot, \cdot]$ is the graded commutator on $End(C_{\alpha, \alpha}(\mathfrak{g},\mathfrak{g}))$.
\end{lem}
\begin{proof}
Let $\varphi \in  C_{\alpha, \alpha}^{q}(\mathfrak{g},\mathfrak{g})$, $\psi \in C_{\alpha, \alpha}^{p}(\mathfrak{g},\mathfrak{g})$  $\xi \in C_{\alpha, \alpha}^{r}(\mathfrak{g},\mathfrak{g})$, $X_{1},...,X_{p+q+r} \in \wedge^{2}{\mathfrak{g}}$ and $x \in \mathfrak{g}$
\begin{align*}[j_{\varphi}^{\alpha}, j_{\psi}^{\alpha}](\xi)(X_1, X_2, ..., X_{p+q+r}, x)&=(j_{\varphi}^{\alpha}(j_{\psi}^{\alpha}\xi)-(-1)^{pq} j_{\psi}^{\alpha}(j_{\varphi}^{\alpha}\xi))(X_1, X_2, ..., X_{p+q+r}, x)\\
&=D_{1} -(-1)^{pq} D_{2},
\end{align*}
where
$$D_1=j_{\varphi}^{\alpha}(j_{\psi}^{\alpha}\xi)(X_1, X_2, ..., X_{p+q+r}, x)\ \textrm{ and} \ D_2=j_{\psi}^{\alpha}(j_{\varphi}^{\alpha}\xi))(X_1, X_2, ..., X_{p+q+r}, x).$$
 For each subset $J=\{j_1,\cdots,j_{q}\}_{j_1<\cdots<j_{q}}\subset N\triangleq \{1,2,\cdots,p+q+r\}$,
  let $I=\{i_1,\cdots,i_{p+r}\}_{i_1<\cdots<i_{p+r}}=N\setminus J$ and
   $L=\{l_1,\cdots,l_p\}_{l_1<\cdots<l_p}\subset I\triangleq \{i_{1},i_{2},\cdots,i_{p+r}\}$,~let $H=\{h_1,\cdots,h_{r}\}_{h_1<\cdots<h_{r}}=I\setminus L$,
we have for $D_1$
{\small\begin{align*}
&j_{\varphi}^{\alpha}(j_{\psi}^{\alpha}\xi)(X_1, X_2, ..., X_{p+q+r}, x)=(j_{\psi}^{\alpha}(\xi))\circ_\alpha\varphi(X_1, X_2, ..., X_{p+q+r}, x)\\
&=\sum_{J, j_{q}<i_{k+1}\leq p+q+r} (-1)^{(J,I)}(j_{\psi}^{\alpha}(\xi))(\alpha ^{q}(X_{i_{1}}), ...,\alpha ^{q}(X_{i_{k}}), \varphi(X_{j_{1}}, ..., X_{j_{q}},~) \bullet_\alpha X_{i_{k+1}},\alpha ^{q}(X_{i_{k+2}}),..., \alpha ^{q}(X_{i_{p+r}}), \alpha ^{q}(x) )\\
&+\sum_{J} (-1)^{(J,I)}(-1)^{p+r}(j_{\psi}^{\alpha}(\xi))(\alpha ^{q}(X_{i_{1}}), ...,\alpha ^{q}(X_{i_{p+r}}), \varphi(X_{j_{1}}, ..., X_{j_{q}}, x))\\
  &=\sum_{J, j_{q}<i_{k+1}\leq p+q+r} \displaystyle \sum_{\substack {{L, l_{p}<h_{m+1}\leq i_{p+r}, }\\{i_{k+1}=h_{n}, n\leq m }}}
 (-1)^{(J,I)}(-1)^{(L,H)}\xi(\alpha^{p+q}(X_{h_1}), ..., \alpha^{p+q}(X_{h_{n-1}}), \alpha^{p}(\varphi(X_{j_{1}},..., X_{j_{q}},  ~) \bullet_\alpha X_{i_{k+1}}),\\
 &...,\alpha^{p+q}(X_{h_m}),\psi(\alpha^{q}(X_{l_1}), ..., \alpha^{q}(X_{l_p}),  ~) \bullet_\alpha \alpha^{q}(X_{h_{m+1}}),\alpha^{p+q}(X_{h_{m+2}}),..., \alpha^{p+q}(X_{h_{r}}), \alpha^{p+q}(x ))\\
 &+\sum_{J, j_{q}<i_{k+1}\leq p+q+r} \displaystyle \sum_{\substack {{L, l_{p}<h_{m+1}\leq i_{p+r}, }\\{i_{k+1}=h_{m+1} }}}
 (-1)^{(J,I)}(-1)^{(L,H)}\xi(\alpha^{p+q}(X_{h_1}),...,\alpha^{p+q}(X_{h_m}), \psi(\alpha^{q}(X_{l_1}), ..., \alpha^{q}(X_{l_p}),~)\\ &\bullet_\alpha(\varphi(X_{j_{1}},..., X_{j_{q}}, ~) \bullet_\alpha X_{i_{k+1}}),\alpha^{p+q}(X_{h_{m+2}}),..., \alpha^{p+q}(X_{h_{r}}), \alpha^{p+q}(x) )\\
 &+\sum_{J, j_{q}<i_{k+1}\leq p+q+r} \displaystyle \sum_{\substack {{L, l_{p}<h_{m+1}\leq i_{p+r}, }\\{i_{k+1}=h_{n}, n>m+1 }}}
 (-1)^{(J,I)}(-1)^{(L,H)}\xi(\alpha^{p+q}(X_{h_1}), ...,\alpha^{p+q}(X_{h_m}), \psi(\alpha^{q}(X_{l_1}), ..., \alpha^{q}(X_{l_p}), ~) \bullet_\alpha\\ &\alpha^{p+q}(X_{h_{m+1}}),\alpha^{p+q}(X_{h_{m+2}}),...,\alpha^{p+q}(X_{h_{n-1}}) ,\alpha^{p}(\varphi(X_{j_{1}},..., X_{j_{q}},  ~) \bullet_\alpha X_{i_{k+1}}),\alpha^{p+q}(X_{h_{n+1}},...,\alpha^{p+q}(X_{h_{r}}), \alpha^{p+q}(x ))\\
 &+\sum_{J, j_{q}<i_{k+1}\leq p+q+r} \displaystyle \sum_{\substack {{L, l_{p}<h_{m+1}\leq i_{p+r}, }\\{i_{k+1}=l_{s}, s\leq p }}}
 (-1)^{(J,I)}(-1)^{(L,H)}\xi(\alpha^{p+q}(X_{h_1}), ...,\alpha^{p+q}(X_{h_m}), \psi(\alpha^{q}(X_{l_1}), ...,\alpha^{q}( X_{l_{s-1}}),\\
 &\varphi(X_{j_{1}},..., X_{j_{q}},X_{j_{q}}, ~) \bullet_\alpha X_{i_{k+1}},\alpha^{q}(X_{l_{s+1}}),...,\alpha^{q}(X_{l_p}), ~) \bullet_\alpha \alpha^{q}(X_{h_{m+1}}),\alpha^{p+q}(X_{h_{m+2}}),..., \alpha^{p+q}(X_{h_{r}}), \alpha^{p+q}(x) )\\
 &+\sum_{J, j_{q}<i_{k+1}\leq p+q+r} \displaystyle \sum_{\substack {{L,i_{k+1}=h_{n}}}}
  (-1)^{(J,I)} (-1)^{(L,H)}(-1)^{r}\xi(\alpha^{p+q}(X_{h_{1}}), ...,\alpha^{p+q}(X_{h_{n-1}}),\alpha^{p}(\varphi(X_{j_{1}},..., X_{j_{q}},~) \bullet_\alpha X_{i_{k+1}}),\\
  &\alpha^{p+q}(X_{h_{n+1}}),...,\alpha^{p+q}(X_{h_{r}}),\psi(\alpha^{q}(X_{l_{1}}), ...,\alpha^{q}( X_{l_{p}}), \alpha^{q}(x)))\\
  &+\sum_{J, j_{q}<i_{k+1}\leq p+q+r} \displaystyle \sum_{\substack {{L,i_{k+1}=l_{s} }}}
  (-1)^{(J,I)}(-1)^{(L,H)}(-1)^{r}\xi(\alpha^{p+q}(X_{h_{1}}), ...,\alpha^{p+q}(X_{h_{r}}),\psi(\alpha^{q}(X_{l_{1}}), ..., \alpha^{q}(X_{l_{s-1}}),\\
  &\varphi(X_{j_{1}},..., X_{j_{q}} ,~) \bullet_\alpha X_{i_{k+1}},\alpha^{q}(X_{l_{s+1}}),...,\alpha^{q}(X_{l_{p}}), \alpha^{q}(x)))\\
  &+\sum_{J} \sum_{L, l_{p}<h_{m+1}\leq i_{p+r}}(-1)^{(J,I)}(-1)^{p+r} (-1)^{(L,H)}\xi(\alpha^{p+q}(X_{h_1}), ...,\alpha^{p+q}(X_{h_m}), \psi(\alpha^{q}(X_{l_1}), ..., \alpha^{q}(X_{l_p}),~)\bullet_\alpha (\alpha^{q}( X_{h_{m+1}}),\\
  &\alpha^{p+q}(X_{h_{m+2}}),...,\alpha^{p+q}( X_{h_{r}}), \alpha^{p}(\varphi(X_{j_{1}}, ..., X_{j_{q}}, x)))\\
&+\sum_{J} \sum_{L}(-1)^{(J,I)}(-1)^{p+r} (-1)^{(L,H)}(-1)^{r}\xi(\alpha^{p+q}(X_{h_{1}}), ...,\alpha^{p+q}(X_{h_{r}}),\psi(\alpha^{q}(X_{l_{1}}), ..., \alpha^{p+q}(X_{l_{p}}), \varphi(X_{j_{1}}, ..., X_{j_{q}}, x))).
  \end{align*}}
 Similarly one can compute  $D_{2}.$

Let $A=\{a_1,a_2,...,a_{p+q}\}_{a_1<a_2<...<a_{p+q}}\subseteq N=\{1,..., p+q+r\}$, $H=\{h_1,...,h_r\}_{h_1<h_2<...<h_r}=N\setminus A$, and $J=\{j_1, ..., j_q\}_{j_1<...<j_q}\subseteq A$ and $L=\{l_1,...,l_p\}_{l_1<...<l_p}=A\setminus J$. We have 
{\small\begin{align*}
 &j_{[\varphi, \psi]^{3HL}}(\xi)(X_1,...,X_{p+q+r},x)\\
 &=\sum_{A, a_{p+q}<h_{m+1}\leq p+q+r}(-1)^{(A,H)}\xi(\alpha^{p+q}(X_{h_1}), ...,\alpha^{p+q}(X_{h_m}),[\varphi,\psi]_{\alpha}^{3HL}(X_{a_1}, ..., X_{a_{p+q}},~)\bullet_\alpha X_{h_{m+1}},\alpha^{p+q}(X_{h_{m+2}}),...,\\
  &\alpha^{p+q}(X_{h_{r}}),\alpha^{p+q}(x))\\
 &+\sum_{A}(-1)^{(A,H)}(-1)^{r}\xi(\alpha^{p+q}(X_{h_1}), ...,\alpha^{p+q}(X_{h_r}),[\varphi,\psi]_{\alpha}^{3HL}(X_{a_1}, ..., X_{a_{p+q}},x))\\
 &= \sum_{A, a_{p+q}<h_{m+1}\leq p+q+r}\sum_{L, l{p}<j_{t+1}\leq a_{p+q}}(-1)^{pq} (-1)^{(A,H)}(-1)^{(L,J)}\xi(\alpha^{p+q}(X_{h_{1}}), ...,\alpha^{p+q}(X_{h_{m}}),\varphi(\alpha^{p}(X_{j_{1}}), ..., \alpha^{p}(X_{j_{t}}),\\
  &\psi(X_{l_{1}},..., X_{l_{p}},~) \bullet_\alpha X_{j_{t+1}}, \alpha^{p}(X_{j_{t+2}})...,\alpha^{p}(X_{j_{q}}),~)\bullet_\alpha \alpha^{p}(X_{h_{m+1}}),\alpha^{p+q}(X_{h_{m+2}}),...,\alpha^{p+q}(X_{h_{r}}),\alpha^{p+q}(x))\\
 &-\sum_{A, a_{p+q}<h_{m+1}\leq p+q+r}\sum_{J, j{q}<l_{s+1}\leq a_{p+q}} (-1)^{(A,H)}(-1)^{(J,L)}\xi(\alpha^{p+q}(X_{h_{1}}), ...,\alpha^{p+q}(X_{h_{m}}),\psi(\alpha^{q}(X_{l_{1}}), ..., \alpha^{q}(X_{l_{s}}),\\
 &\varphi(X_{j_{1}},..., X_{j_{q}},~) \bullet_\alpha X_{l_{s+1}}, \alpha^{q}(X_{l_{s+2}})...,\alpha^{q}(X_{l_{p}}),~)\bullet_\alpha \alpha^{q}(X_{h_{m+1}}),...,\alpha^{p+q}(X_{h_{r}}),\alpha^{p+q}(x))\\
  &+\sum_{A, a_{p+q}<h_{m+1}\leq p+q+r}\sum_{L} (-1)^{(A,H)}(-1)^{(L,J)}(-1)^{pq}(-1)^{q}\xi(\alpha^{p+q}(X_{h_{1}}), ...,\alpha^{p+q}(X_{h_{m}}),\varphi(\alpha^{p}(X_{j_{1}}), ..., \alpha^{p}(X_{j_{q}},~)\bullet_{\alpha}\\
 &(\psi(X_{l_{1}},..., X_{l_{p}},~) \bullet_\alpha X_{h_{m+1}}), \alpha^{p+q}(X_{h_{m+2}}),...,\alpha^{p+q}(X_{h_{r}}),\alpha^{p+q}(x))\\
&-\sum_{A, a_{p+q}<h_{m+1}\leq p+q+r}\sum_{J} (-1)^{(A,H)}(-1)^{(J,L)}(-1)^{pq}(-1)^{P}\xi(\alpha^{p+q}(X_{h_{1}}), ...,\alpha^{p+q}(X_{h_{m}}),\psi(\alpha^{q}(X_{l_{1}}), ..., \alpha^{q}(X_{l_{p}},~)\bullet_{\alpha}\\
 &(\varphi(X_{j_{1}},..., X_{j_{q}},~) \bullet_\alpha X_{h_{m+1}}), \alpha^{p+q}(X_{h_{m+2}}),...,\alpha^{p+q}(X_{h_{r}}),\alpha^{p+q}(x))\\
 &+\sum_{A}\sum_{L, l_{p}<j_{t+1}\leq a_{p+q}} (-1)^{r}(-1)^{pq}(-1)^{(A,H)}(-1)^{(L,J)}(\xi(\alpha^{p+q}(X_{h_{1}}), ...,\alpha^{p+q}(X_{h_{r}}),\varphi(\alpha^{p}(X_{j_{1}}), ...,\alpha^{p}( X_{j_{t}}),\psi(X_{l_{1}},..., X_{l_{p}},~)\\
  &\bullet_{\alpha}X_{j_{t+1}}, \alpha^{p}(X_{j_{t+2}}),...,\alpha^{p}(X_{j_{q}}),\alpha^{p}(x)))\\
  &+\sum_{A}\sum_{L} (-1)^{r}(-1)^{q}(-1)^{pq}(-1)^{(A,H)}(-1)^{(L,J)}(\xi(\alpha^{p+q}(X_{h_{1}}), ...,\alpha^{p+q}(X_{h_{r}}),\varphi(\alpha^{p}(X_{j_{1}}), ..., \alpha^{p}(X_{j_{q}}),\psi(X_{l_{1}},..., X_{l_{p}},x)).\\
  &-\sum_{A}\sum_{J, J_{q}<l_{s+1}\leq a_{p+q}} (-1)^{r}(-1)^{(A,H)}(-1)^{(J,L)}(\xi(\alpha^{p+q}(X_{h_{1}}), ...,\alpha^{p+q}(X_{h_{r}}),\psi(\alpha^{q}(X_{l_{1}}), ..., \alpha^{q}(X_{l_{s}}),\\
  &\varphi(X_{j_{1}},..., X_{j_{q}},~)\bullet_{\alpha} X_{l_{s+1}}, \alpha^{q}(X_{l_{s+2}}),...,\alpha^{q}(X_{l_{p}}),\alpha^{q}(x)))\\
  &-\sum_{A}\sum_{J} (-1)^{p}(-1)^{r}(-1)^{(A,H)}(-1)^{(J,L)}(\xi(\alpha^{p+q}(X_{h_{1}}), ...,\alpha^{p+q}(X_{h_{r}}),\psi(\alpha^{q}(X_{l_{1}}), ..., \alpha^{q}(X_{l_{p}}),\varphi(X_{j_{1}},..., X_{j_{q}},x)).
  \end{align*}}
By a straightforward verification, we obtain  $D_1 -(-1)^{pq}D_2=j_{[\varphi, \psi]^{3HL}}^{\alpha}(\xi)(X_1,...,X_{p+q+r},x).$
Hence the proof.
\end{proof}

\begin{thm}\label{thmdalgebreDeLieGradue}
The pair  $(C_{\alpha,\alpha}(\mathfrak{g},\mathfrak{g}), [\cdot,\cdot]^{3HL})$ is a graded Lie algebra.
\end{thm}
\begin{proof}

 Let $\varphi \in  C_{\alpha, \alpha}^{q}(\mathfrak{g},\mathfrak{g})$, $\psi \in C_{\alpha, \alpha}^{p}(\mathfrak{g},\mathfrak{g})$ and $\phi \in C_{\alpha, \alpha}^{r}(\mathfrak{g},\mathfrak{g})$.
\begin{enumerate}
\item skew-symmetry
\begin{align*}[\varphi,\psi]^{3HL} &=(-1)^{pq}j_\varphi^\alpha(\psi)-j_\psi^\alpha(\varphi) \\ \ & =(-1)^{pq+1}\big((-1)^{pq}j_{\psi}^\alpha (\varphi)-j_\varphi^\alpha(\psi)\big) \\ \ & =-(-1)^{pq}[\psi,\varphi]^{3HL}.
\end{align*}
\item Graded Jacobi identity
\begin{align*}
\circlearrowleft_{\varphi,\psi,\phi}(-1)^{qr}\big[\varphi,[\psi,\phi]^{3HL}\big]^{3HL} & =
(-1)^{qp}j_{[\psi,\phi]^{3HL}}(\varphi)-(-1)^{qr}j_\varphi^\alpha([\psi,\phi]^{3HL})\\ \ &
+(-1)^{pr}j_{[\phi,\varphi]^{3HL}}(\psi)-(-1)^{pq} j_\psi^\alpha([\phi,\varphi]^{3HL})
\\ \ &  (-1)^{rq}j_{[\varphi,\psi]^{3HL}}(\phi)-(-1)^{rp} j_\phi^\alpha([\varphi,\psi]^{3HL})\\
 \ & =(-1)^{qp}j_{[\psi,\phi]^{3HL}}(\varphi)-(-1)^{qr}j_\varphi^\alpha((-1)^{rp}j_\phi^\alpha(\psi)-j_\psi^\alpha (\phi))\\ \ & (-1)^{pr}j_{[\phi,\varphi]^{3HL}}(\psi)-(-1)^{pq}j_\psi^\alpha((-1)^{qr}j_\varphi^\alpha(\phi)-j_\phi^\alpha(\varphi)) \\ \ & (-1)^{rq}j_{[\varphi,\psi]^{3HL}}(\phi)-(-1)^{rp}j_\phi^\alpha ((-1)^{qp}j_\psi^\alpha(\varphi)-j_\varphi^\alpha(\psi)).
\end{align*}
                Organizing  these terms leads to
\begin{align*}
\circlearrowleft_{\varphi,\psi,\phi}(-1)^{qr}\big[\varphi,[\psi,\phi]^{3HL}\big]^{3HL} & =(-1)^{pq}j_{[\psi,\phi]^{3HL}}(\varphi)+(-1)^{pq}\big(j_\psi^\alpha(j_\phi^\alpha(\varphi))-(-1)^{rp}j_\phi^\alpha (j_\psi^\alpha(\varphi))\\ \ &
  +(-1)^{rp}j_{[\phi,\varphi]^{3HL}}(\psi) +(-1)^{rp}\big(j_\phi^\alpha (j_\varphi^\alpha(\psi))-(-1)^{qr}j_\varphi^\alpha(j_\phi^\alpha(\psi)) \\ \ &
  +(-1)^{qr}j_{[\varphi,\psi]^{3HL}}(\phi)+(-1)^{qr}\big(j_\varphi^\alpha (j_\psi^\alpha (\phi))-(-1)^{qp}j_\psi^\alpha (j_\varphi^\alpha(\phi))\\ \ &=(-1)^{pq}\big([j_\psi^\alpha,j_\phi^\alpha]+j_{[\psi,\phi]^{3HL}}\big)(\varphi)\\ \ & +(-1)^{rp}\big([j_\phi^\alpha,j_\varphi^\alpha]+j_{[\phi,\varphi]^{3HL}}\big)(\psi)\\ \ & +(-1)^{qr}\big([j_\varphi^\alpha,j_\psi^\alpha]+j_{[\varphi,\psi]^{3HL}}\big)(\phi).
  \end{align*}
  Using the previous lemma we get
  $$\circlearrowleft_{\varphi,\psi,\phi}(-1)^{qr}\big[\varphi,[\psi,\phi]^{3HL}]^{3HL}=0.$$
  \end{enumerate}
  \end{proof}
\begin{remark}The pair $(C_{\alpha,\alpha}(\mathfrak{g},\mathfrak{g}), \circ_\alpha)$ is a right symmetric graded algebra.
\end{remark}

The previous structure of graded Lie algebra is useful to describe   3-Hom-Lie algebra structures as well as coboundary operators.
\begin{cor}\label{lem de crochetAunestucturedeAlgebreGradue}
The maps $\pi : \wedge^3 \mathfrak{g} \longrightarrow \mathfrak{g} $ and $\alpha:\mathfrak{g} \longrightarrow \mathfrak{g}$ define a   $3$-Hom-Lie structure if and only if $[\pi, \pi]^{3HL} = 0$.
\end{cor}
Let $\mathfrak{g}$ be a   $3$-Hom-Lie algebra. Given $x_{1}, x_{2} \in \mathfrak{g}$, define $ad : \wedge^2 \mathfrak{g}\longrightarrow gl(\mathfrak{g})$ by
$$ad_{x_{1},x_{2}}y=[x_{1},x_{2},y].$$
Then, the pair  $(ad,\alpha)$ defines a representation of the   $3$-Hom-Lie algebra $\mathfrak{g}$ on itself, which we call adjoint representation of $\mathfrak{g}$.
The coboundary operator associated to this representation is denoted by $\delta_{\mathfrak{g}}$.

 \begin{cor}
If $\pi:\wedge^3\mathfrak{g}\longrightarrow \mathfrak{g}$ is a   $3$-Hom-Lie bracket, then we have
\begin{equation}
[\pi,\varphi]^{3HL}=\delta_{g}(\varphi),\quad \forall\varphi\in C_{\alpha, \alpha}^p(\mathfrak{g},\mathfrak{g}),p\geq0.
\end{equation}
\end{cor}

\section{Generalized representations of   $3$-Hom-Lie algebras}

In this section, we provide  the Hom-type version  of generalized representation of a  $3$-Lie algebras introduced in \cite{LuiMakhloufSheng}. First, we show that a representation of a   $3$-Hom-Lie algebra will give rise to a canonical structure.

Let  $\mathfrak{g}$ be a   $3$-Hom-Lie algebra and $V$ be a vector space. Let   $\rho : \wedge^2 \mathfrak{g} \longrightarrow gl(V )$ be a linear map. Then, it induces a linear map $\overline{\rho} : \wedge^3(\mathfrak{g}\oplus V)\longrightarrow \mathfrak{g}\oplus V$
defined by
\begin{equation} \label{stuctureDeRhoBAR}\overline{\rho}(x + u, y + v, z + w) = \rho(x, y)(w) + \rho(y, z)(u) + \rho(z, x)(v), \forall x, y, z \in \mathfrak{g},  u, v,w \in  V.\end{equation}
Consider the graded Lie algebra given in Theorem \ref{thmdalgebreDeLieGradue} associated to the vector space $\mathfrak{g}\oplus V$.
\begin{proposition}
A linear map $\rho : \wedge^2 \mathfrak{g} \longrightarrow gl(V )$ is a representation on a vector space $V$ of the   $3$-Hom-Lie algebra $\mathfrak{g}$ with respect to $A \in gl(V)$  if and only if $\pi+ \overline{\rho}$ is a canonical structure in the graded Lie algebra associated to $\mathfrak{g}\oplus V $, i.e.
$$[\pi+ \overline{\rho},\pi+ \overline{\rho}]^{3HL} = 0.$$
\end{proposition}
\begin{proof}
 By Proposition \ref{lem:semidirectp}, $\rho:\wedge^2 \mathfrak{g}\longrightarrow gl(V)$ is a representation of $\mathfrak{g}$ if and only if $\mathfrak{g}\oplus V$ is a $3$-Hom-Lie algebra, where the   $3$-Hom-Lie structure is exactly given by
 \begin{eqnarray*}
[x+u,y+v,z+w]_{\rho}&=&[x,y,z]+\rho(x,y)(w)+\rho(y,z)(u)+\rho(z,x)(v)\\
&=&(\pi+\bar{\rho})(x+u,y+v,z+w),
\end{eqnarray*}
and $\alpha_{\mathfrak{g}\oplus V}=\alpha + A$. Thus, by Lemma \ref{lem de crochetAunestucturedeAlgebreGradue},  $\rho:\wedge^2 \mathfrak{g} \longrightarrow gl(V)$ is a representation of $\mathfrak{g}$ if and only if $\pi+\bar{\rho}$ is a canonical structure.
\end{proof}

The  concept of representation of  $3$-Lie algebras introduced by Liu, Makhlouf and   Sheng (\cite{LuiMakhloufSheng}) is generalized to Hom-type algebras as follows.
\begin{df}\label{DefDerepresentationGeneralise}
 A {\bf generalized representation }  of a   $3$-Hom-Lie algebra $(\mathfrak{g},[\cdot,\cdot,\cdot],\alpha)$ with respect to $A \in gl(V)$ consists of linear maps $\rho: \wedge^2 \mathfrak{g} \longrightarrow gl(V )$, $ \nu : \mathfrak{g} \longrightarrow Hom(\wedge ^2 V, V )$,  such that
\begin{equation}\label{cochetd eLIGRADUEilifih ipiWILrhoWILnu}[\pi+\overline{\rho}+\overline{\nu},\pi+\overline{\rho}+\overline{\nu}]^{3HL} = 0, \end{equation}
where $\overline{\nu} :\wedge^3(\mathfrak{g}\oplus V) \longrightarrow (\mathfrak{g}\oplus V)$ is induced by $\nu$ via
$$\overline{\nu}(x + u, y + v, z + w) = \nu(x)(v\wedge w) + \nu(y)(w \wedge u) + \nu(z)(u  \wedge v), \forall x, y, z \in \mathfrak{g}, u, v,w \in V.$$
We will refer to a generalized representation by  $(V ; \rho, \nu,A)$.
\end{df}

\begin{remark}

 If $ \nu = 0$,~ then we recover the usual definition of a representation of a $3$-Hom-Lie algebra on a vector space $V$. If the dimension of the vector space$ V$ is 1, then $ \nu$ must be zero. In this case, we only have the usual representation.

\end{remark}
Given linear maps $ \rho : \wedge^2 \mathfrak{g}\longrightarrow End(V )$,  $ \nu : \mathfrak{g}\longrightarrow  Hom(\wedge^2 V,V)$,  and $A: V\longrightarrow V$ define a trilinear bracket
operation on $\mathfrak{g}\oplus V$ by
\begin{equation}\label{eqdeproduitsemidirectegeneralise}[x + u, y + v, z + w]_{(\rho,\nu)} = [x, y, z] + \rho(x, y)(w) + \rho(y, z)(u) + \rho(z, x)(v)
+\nu(x)(v \wedge w) + \nu(y)(w \wedge u) + \nu(z)(u \wedge v).\end{equation}

\begin{thm}\label{Theoremde repgeneraliseAvecPRODUITSEMIdirecteGNERALISE} Let $(\mathfrak{g}, [\cdot,\cdot,\cdot],\alpha)$ be a   $3$-Hom-Lie algebra and $(V ; \rho, \nu,A)$ a generalized representation of $\mathfrak{g}$ with respect to $A$. Then $(\mathfrak{g}\oplus V, [\cdot,\cdot,\cdot]_{(\rho,\nu)},\alpha_{\mathfrak{g}\oplus V}=\alpha+A)$ is a   $3$-Hom-Lie algebra, where $[\cdot,\cdot,\cdot]_{(\rho,\nu)}$is given by \eqref{eqdeproduitsemidirectegeneralise}.\\
We call the    $3$-Hom-Lie algebra $(\mathfrak{g}\oplus V, [\cdot,\cdot, \cdot]_{(\rho,\nu)},\alpha_{\mathfrak{g}\oplus V}=\alpha+A)$ \textbf{the generalized semidirect product }of $\mathfrak{g}$ and $V$.
\end{thm}
\begin{proof}
It follows from
$[x + u, y + v, z + w](\rho,\nu) = (\pi +\overline{\rho} + \overline{\nu})(x + u, y + v, z + w)$
and Lemma \ref{lem de crochetAunestucturedeAlgebreGradue}.
\end{proof}
In the following, we give a characterization of a generalized representation of a $3$-Hom-Lie algebra.
\begin{proposition}
Let  $ \rho:\wedge^2 \mathfrak{g} \longrightarrow End(V)$, $ \nu : \mathfrak{g} \longrightarrow  Hom(\wedge^2 V, V )$, $A: V \longrightarrow V$  be linear maps. They give rise to a generalized representation of a   $3$-Hom-Lie algebra $\mathfrak{g}$ with respect to $A$  if and only if for all $ x_{i}\in \mathfrak{g}$,
$v_{j}\in V$, the following equalities hold:
\begin{align}
\rho(\alpha(x_{1}),\alpha(x_{2}))\rho(x_{3},x_{4}) =& \rho([x_{1},x_{2},x_{3}],\alpha(x_{4}))\circ A -\rho([x_{1},x_{2},x_{4}],\alpha(x_{3}))\circ A\nonumber\\&+\rho(\alpha(x_{3}),\alpha(x_{4}))\rho(x_{1},x_{2}), \label{eq1deRepresenGenerelise}\\
\rho([x_{1},x_{2},x_{3}],\alpha(x_{4}))\circ A =&\rho(\alpha(x_{2}),\alpha(x_{3}))\rho(x_{1},x_{4})+\rho(\alpha(x_{3}),\alpha(x_{1}))\rho(x_{2},x_{4})\nonumber\\&+\rho(\alpha(x_{1}),\alpha(x_{2}))\rho(x_{3},x_{4}), \label{eq2deRepresenGenerelise}\\
\rho(\alpha(x_{1}),\alpha(x_{2}))\nu(x_{3})(v_{1},v_{2})=&\nu([x_{1},x_{2},x_{3}])(A(v_{1}),A(v_{2}))+\nu(\alpha(x_{3}))(\rho(x_{1},x_{2})v_{1},A(v_{2}))\nonumber\\&+\nu(\alpha(x_{3}))(A(v_{2}),\rho(x_{1},x_{2})v_{1}), \label{eq3deRepresenGenerelise}\\
\nu(\alpha(x_{1}))(A(v_{1}),\rho(x_{2},x_{3})v_{2})=&\nu(\alpha(x_{3}))(A(v_{2}),\rho(x_{2},x_{1})v_{1})+\nu(\alpha(x_{2}))(\rho(x_{3},x_{1})v_{1},A(v_{2}))\nonumber\\&+\rho(\alpha(x_{2}),\alpha(x_{3}))\nu(x_{1})(v_{1},v_{2}), \label{eq4deRepresenGenerelise}\\
\nu(\alpha(x_{1}))(A(v_{1}),\nu(x_{2})(v_{2},v_{3}))=&\nu(\alpha(x_{2}))(\nu(x_{1})(v_{1},v_{2}),A(v_{3}))+\nu(\alpha(x_{2}))(A(v_{2}),\nu(x_{1})(v_{1},v_{3})), \label{eq5deRepresenGenerelise}\\
\nu(\alpha(x_{1}))(\nu(x_{2})(v_{1},v_{2}),A(v_{3}))=&\nu(\alpha(x_{2}))(\nu(x_{1})(v_{1},v_{2}),A(v_{3})).\label{eq6deRepresenGenerelise}
\end{align}\end{proposition}
\begin{proof}
the quadruple $(V; \rho,\nu,A)$ is a generalized representation if and only if $[\pi+ \overline{\rho}+\overline{\nu}, \pi+ \overline{\rho}+\overline{\nu}]^{3HL} = 0$. By straightforward computations,
$$
[\pi+\bar{\rho}+\bar{\nu},\pi+\bar{\rho}+\bar{\nu}]^{3HL}(x_1,x_2,x_3,x_4,v)=0
$$
is equivalent to \eqref{eq1deRepresenGenerelise}; and
$$
[\pi+\bar{\rho}+\bar{\nu},\pi+\bar{\rho}+\bar{\nu}]^{3HL}(x_1,v,x_2,x_3,x_4)=0
$$
is equivalent to \eqref{eq2deRepresenGenerelise}. Other identities  can be proved similarly.  The details are omited.\end{proof}

\begin{remark}
  By \eqref{eq1deRepresenGenerelise} and \eqref{eq2deRepresenGenerelise}, the map $\rho$ in a generalized representation $(V;\rho,\nu,A)$ gives rise to a usual representation in the sense of Definition \ref{def2.2}. Conversely, for any representation $\rho$, $(V;\rho,\nu=0, A)$ is a generalized representation.
\end{remark}

\begin{df}
  Let $(V_1;\rho_1,\nu_1,A_1)$ and $(V_2;\rho_2,\nu_2,A_2)$ be two generalized representations of a $3$-Hom-Lie algebra $(\g,[\cdot,\cdot,\cdot],\alpha)$. They are said to be {\bf equivalent} if there exists an isomorphism of vector spaces $T:V_1\longrightarrow V_2$ such that
  $$
  T\rho_1(x,y)(u)=\rho_2(x,y)(Tu),\quad T\nu_1(x)(u,v)=\nu_2(x)(Tu,Tv),\quad T\circ A_1=A_2\circ T\quad\forall x,y\in\g,~u, v\in V_1.
  $$
  In terms of diagrams, we have
  $$
\xymatrix{
 \wedge^2\g\times V_1 \ar[d]_{ id\times T }\ar[rr]^{\rho_1}
                && V_1  \ar[d]^{T}  \\
 \wedge^2\g\times V_2 \ar[rr]^{\rho_2}
                && V_2  },\quad \xymatrix{
 \g\times \wedge^2 V_1 \ar[d]_{ id\times \wedge^2T }\ar[rr]^{\nu_1}
                && V_1  \ar[d]^{T}  \\
 \g\times \wedge^2 V_2 \ar[rr]^{\nu_2}
                && V_2  },\quad \xymatrix{
 V_1 \ar[d]_{ T }\ar[rr]^{A_1}
                && V_1  \ar[d]^{T}  \\
 V_2 \ar[rr]^{A_2}
                && V_2.  }
$$
\end{df}

In the following, we provide a series of examples to illustrate the new concept of generalized representation and also a procedure to twist a generalized representation along linear maps.
\begin{example}
 Let $\mathfrak{g}$ be an abelian   $3$-Hom-Lie algebra. Define $\rho=0$, and $\nu=\xi\otimes \pi$, where  $\xi\in\g^*$ and $\pi\in Hom(\wedge^2V\otimes V)$ is a Hom-Lie algebra structure on $(V,\pi, A)$. Then $(V;\rho,\nu, A)$ is a generalized representation. In fact,
 since $\mathfrak{g}$ is abelian and $\rho=0$, \eqref{eq1deRepresenGenerelise}-\eqref{eq4deRepresenGenerelise} hold naturally.  Since $\pi$ satisfies the Hom-Jacobi identity, \eqref{eq5deRepresenGenerelise} and \eqref{eq6deRepresenGenerelise} also hold.
\end{example}

\begin{proposition}
Let $(\mathfrak{g},[\cdot,\cdot,\cdot],\alpha)$ be a   $3$-Hom-Lie algebra, $(V,\rho,\nu,A)$ be a generalized representation,  $\beta : \mathfrak{g}\rightarrow \mathfrak{g}$ be an algebra morphism and  $B : V\rightarrow V $ a linear map such that
\begin{align}
& B \circ \rho(x_{1},x_{2})= \rho(\beta(x_{1}),\beta(x_{2}))\circ B.\\
 &B\circ\nu(x)=\nu(\beta(x))\circ(B\otimes B)\\
 &B\circ A=A\circ B.
 \end{align}
Then $(V, \widetilde{\rho},\widetilde{\nu},B)$ is a generalized representation of $3$-Hom-Lie algebra $(\mathfrak{g}, [\cdot,\cdot,\cdot]_{\beta},\beta\circ\alpha)$ where
 $$ [\cdot,\cdot,\cdot]_{\beta}=[\cdot,\cdot,\cdot]\circ \beta^{\otimes 3},
 \; \widetilde{\rho}= B\circ \rho,\;
 \widetilde{\nu}(x)=B\circ\nu(x).$$
\end{proposition}
\begin{proof}
We have to show that $\widetilde{\rho}$ and $\widetilde{\nu}$ satisfy Eqs.\eqref{eq1deRepresenGenerelise}-\eqref{eq6deRepresenGenerelise}.\\
Let $x_1,x_2,x_3 \in \mathfrak{g}$ and $v_1,v_2 \in V$
\begin{align*}
&\widetilde{\rho}(\beta\circ\alpha(x_1),\beta\circ\alpha(x_2))\widetilde{\nu}(x_3)(v_1,v_2)-\widetilde{\nu}([x_1,x_2,x_3]_{\beta})(B\circ A(v_1),B\circ A(v_2))\\
&-\widetilde{\nu}(\beta\circ\alpha(x_3))(\widetilde{\rho}(x_1,x_2)v_1,B\circ A(v_2))-\widetilde{\nu}(\beta\circ\alpha(x_3))(B\circ A(v_2),\widetilde{\rho}(x_1,x_2)v_1)\\
&=B \circ\rho(\beta\circ\alpha(x_1),\beta\circ\alpha(x_2))\circ B\circ\nu(x_3)(v_1,v_2)-B\circ\nu(\beta\circ[x_1,x_2,x_3])\circ(B\otimes B)(A(v_1),A(v_2))\\
&-B\circ\nu(\beta\circ\alpha(x_3))(B\circ\rho(x_1,x_2)v_1,B\circ A (v_2))-B\circ\nu(\beta\alpha(x_3))(B\circ A(v_2),B\circ\rho(x_1,x_2)v_1))\\
&=B^{2}\circ(\rho(\alpha(x_1),\alpha(x_2))\circ\nu(x_3)(v_1,v_2)-\nu([x_1,x_2,x_3])(A(v_1),A(v_2))\\
&-\nu(\alpha(x))(\rho(x_1,x_2)v_1,A(v_2))-\nu(\alpha(x_3))(A(v_2),\rho(x_1,x_2)v_1))\\
&=0.
\end{align*}
\begin{align*}
&\widetilde{\nu}(\beta\circ\alpha(x_1))(B\circ A(v_1),\widetilde{\rho}(x_2,x_3)v_2)-\widetilde{\nu}(\beta\circ\alpha(x_3))(B\circ A(v_2),\widetilde{\rho}(x_2,x_1)v_1)\\
&-\widetilde{\nu}(\beta\circ\alpha(x_2))(\widetilde{\rho}(x_3,x_1)v_1,B\circ A(v_2))-\widetilde{\rho}(\beta\circ\alpha(x_2),\beta\circ\alpha(x_3))\widetilde{\nu}(x_1)(v_1,v_2)\\
&=B\circ\nu(\beta\circ\alpha(x_1))(B\circ A(v_1),B\circ\rho(x_2,x_3)v_2)-B\circ\nu(\beta\circ\alpha(x_3))(B\circ A(v_2),B\circ\rho(x_2,x_1)v_1)\\
&-B\circ\nu(\beta\circ\alpha(x_2))(B\circ\rho(x_3,x_1)v_1,B\circ A(v_2))-B\circ\rho(\beta\circ\alpha(x_2),\beta\circ\alpha(x_3))\circ B\circ \nu(x_1)(v_1,v_2)\\
&=B^{2}\circ(\nu(\alpha(x_1))(A(v_1),\rho(x_2,x_3)v_2)-\nu(\alpha(x_3))(v_2,\rho(x_2,x_1)v_1)\\
&-\nu(\alpha(x_2))(\rho(x_3,x_1)v_1,A(v_2))-\rho(\alpha(x_2),\alpha(x_3))\circ \nu(x_1)(v_1,v_2))\\
&=0.
\end{align*}
Then identities   \eqref{eq3deRepresenGenerelise} and  \eqref{eq4deRepresenGenerelise} are proved.
One similarly proves   identities \eqref{eq5deRepresenGenerelise} and  \eqref{eq6deRepresenGenerelise}.
\end{proof}

\begin{cor}
Let $(\mathfrak{g},[\cdot,\cdot,\cdot])$ be a $3$-Lie algebra,  $(V,\rho,\nu)$ be a generalized representation, $\alpha : \mathfrak{g}\rightarrow \mathfrak{g}$ be an algebra morphism and  $A : V\rightarrow V $ be a linear map  such that for all $x_1,x_2\in  \mathfrak{g}$ and $v_1,v_2\in V$,
\begin{align}
& A \circ \rho(x_{1},x_{2})= \rho(\alpha(x_{1}),\alpha(x_{2}))\circ A.\\
 &A\circ\nu(x)(v_1,v_2)=\nu(\alpha(x))\circ(A\otimes A)(v_1,v_2).
 \end{align}
Then $(V, \widetilde{\rho}:=A\circ \rho,\widetilde{\nu}:=A\circ \nu,A)$ is a generalized representation of   $3$-Hom-Lie algebra $(L, [\cdot,\cdot,\cdot]_{\alpha}:=\alpha\circ [\cdot,\cdot,\cdot],\alpha) $.
\end{cor}
\begin{example}
Let $\mathfrak{g}$ be the $3$-dimensional $3$-Lie algebra defined with respect to a  basis $\{e_1,e_2,e_3\}$ by the skew-symmetric bracket   $[e_1,e_2,e_3]= e_1$.
 Let $V$ be a $2$-dimensional vector space and $\{v_1,v_2\}$ its basis. We have a representation defined by  the  following  maps $(\rho,\mu)$,  given with respect to previous bases by
\begin{eqnarray*}    &\rho(e_1,e_2)(v_1)=0,  \;
\rho(e_1,e_2)(v_2)= v_1 ,   \;
 \rho(e_1,e_3)(v_1)=0, \\&
  \rho(e_1,e_3)(v_2)=r_1 v_1, \;
 \rho(e_2,e_3)(v_1)= v_1,    \;  \rho(x_e,e_3)(v_2)=r_2 v_1,
 \\&
 \nu(e_1)(v_1,v_2)=0,  \;
\nu(e_2)(v_1,v_2)=sv_1,  \;
\nu(e_3)(v_1,v_2)=s r_1v_1,
\end{eqnarray*}
where $r_1,r_2,s$ are parameters in $\mathbb{K}$.

Let $\alpha: \mathfrak{g}\longrightarrow \mathfrak{g}$ be  a algebra morphism and $A \in gl (V)$ defined respectively by:
\begin{align*}
&\alpha(e_1)=\lambda e_1, \alpha(e_2)=e_2, \alpha(e_3)=e_3,\\
& A(v_1)=\lambda v_1, A(v_2)=r_2 v_1+ v_2,
\end{align*}
where $\lambda$ is a parameter in $\mathbb{K}$.
They satisfy
\begin{align}
& A \circ \rho(x_{1},x_{2})= \rho(\alpha(x_{1}),\alpha(x_{2}))\circ A.\\
 &A\circ\nu(x)(u_1,u_2)=\nu(\alpha(x))\circ(A\otimes A)(u_1,u_2).
 \end{align}
where $x, x_1,x_2$ are in $\mathfrak{g}$ and $u_1,u_2$ in $V$.

Then, using the Twist procedure, $(V; \widetilde{\rho},\widetilde{\nu},A)$ is a generalized representation of the   $3$-Hom-Lie algebra $(\mathfrak{g}, [\cdot,\cdot,\cdot]_{\alpha},\alpha)$. More precisely, we have
\begin{eqnarray*}
  [e_1,e_2,e_3]_{\alpha}&=&[\alpha(e_1),\alpha(e_2),\alpha(e_3)]=\lambda e_1,
  \end{eqnarray*}
  \begin{eqnarray*}
 &  \widetilde{\rho}(e_1,e_2)(v_1)=0,  \;
\widetilde{\rho}(e_1,e_2)(v_2)=\lambda v_1 ,   \;
 \widetilde{\rho}(e_1,e_3)(v_1)=0, \\
&  \widetilde{\rho}(e_1,e_3)(v_2)=r_1 r_2 (\lambda-1) v_1- r_1 v_2, \;
 \widetilde{\rho}(e_2,e_3)(v_1)= \lambda v_1,    \;  \widetilde{\rho}(e_2,e_3)(v_2)=r_2 \lambda v_1,
 \\ &
 \widetilde{\nu}(e_1)(v_1,v_2)=0,  \;
\widetilde{\nu}(e_2)(v_1,v_2)=s \lambda v_1,  \;
\widetilde{\nu}(e_3)(v_1,v_2)=s r_1 \lambda v_1.
\end{eqnarray*}
\end{example}

\begin{example}

Let $\mathfrak{g}$ be  the $4$-dimensional  $3$-Lie algebra defined, with respect to a basis $\{e_1,e_2,e_3,e_4\}$,  by the skew-symmetric brackets
$$[e_1,e_2,e_4]=e_3,\;[e_1,e_3,e_4]=e_2, \; [e_2,e_3,e_4]=e_1. $$
 Every generalized representation $(V;\rho,\nu)$, on a 2-dimensional vector space  $V$  with trivial $\rho$, of $\g$ is given by one of the following maps $\nu$ defined, with respect to a basis  $\{ v_1,v_2\}$ of $V$, by

\begin{enumerate}\item $  \nu(e_1)(v_1,v_2)=0,  \;
\nu(e_2)(v_1,v_2)=0,  \; \nu(e_3)(v_1,v_2)=0, \;
\nu(e_4)(v_1,v_2)=s_1 v_1+s_2  v_2,
 $

\end{enumerate}
where $s_1,s_2$ are parameters in $\K$, and $s_1s_2 \neq0$.\\
Let $\alpha : \mathfrak{g\longrightarrow \mathfrak{g}}$ be a  $3$-Lie algebra morphism and $A:V\longrightarrow V$ be a linear map, defined respectively by
$$\alpha(e_1)=a_1 e_1,~ \alpha(e_2)= a_1 e_2, ~\alpha(e_3)= a_1 e_3, ~\alpha(e_4)= \frac{-1}{a_1} e_4 ,$$ $$ A(v_1)=-a_1 v_1,~ A(v_2)= a_2 v_1+ \frac{a_2 s_2- a_1 s_1}{s_1}v_2,$$
where $a_1, a_2,s_1,s_2 $ are parameters in $\K$ such that, $a_1s_1s_2 \neq0$.\\
They satisfy $ A\circ\nu(x)=\nu(\alpha(x))\circ(A\otimes A)$.
Therefore, using the Twist procedure,  $(V; \widetilde{\rho},\widetilde{\nu},A)$ is a generalized representation of the   $3$-Hom-Lie algebra $(L, [\cdot,\cdot,\cdot]_{\alpha},\alpha)$ with trivial $\widetilde{\rho}$. Namely, we have  $$[e_1,e_2,e_4]_{\alpha}=a_1 e_3, \; [e_1,e_3,e_4]_{\alpha}=-a_1 e_2, \; [e_2,e_3,e_4]_{\alpha}=a_1 e_1,$$
and
$$  \widetilde{\nu}(e_1)(v_1,v_2)=0,  \;
\widetilde{\nu}(e_2)(v_1,v_2)=0,  \;
\widetilde{\nu}(e_3)(v_1,v_2)=0, \;
\widetilde{\nu}(e_4)(v_1,v_2)=(a_2 s_2 - a_1 s_1) v_1+(\frac{a_2 s_2}{s_1}- a_1)  v_2. $$

\end{example}

\section{New cohomology complex of $3$-Hom-Lie algebras }
Based on the generalized representations defined in the previous section, we introduce a new type of cohomology for   $3$-Hom-Lie algebras.

Let $(\mathfrak{g}, [\cdot,\cdot,\cdot ],\alpha)$ be a   $3$-Hom-Lie algebra and $(V;\rho,\nu,A)$ be a generalized representation of $\mathfrak{g}$.
We set ${\widetilde{C}}_{\alpha+A,A}^p(\mathfrak{g}\oplus V,V)$ to be the set of $(p+1)$-Hom-cochains, which are defined as a subset of ${C}_{\alpha+A,A}^p(\mathfrak{g}\oplus V,V)$ such that
\begin{equation}\label{eq:C}
{C}_{\alpha+A,A}^p(\mathfrak{g}\oplus V,V)={\widetilde{C}}_{\alpha+A,A}^p(\mathfrak{g}\oplus V,V)\oplus {C}_{A}^p(V,V).
\end{equation}
Elements of ${C}_{\alpha+A,A}^{p}(\mathfrak{g}\oplus V,V)$ are of the form $ \varphi:\wedge^2(\mathfrak{g}\oplus V)\otimes\stackrel{(p \text{ times})}{\cdots }\otimes\wedge^2(\mathfrak{g}\oplus V)\wedge(\mathfrak{g}\oplus V)\longrightarrow V.$

By direct calculation, we have
$$[\pi+\bar{\rho}+\bar{\nu}, {\widetilde{C}}_{\alpha+A,A}^\bullet(\mathfrak{g}\oplus V,V)]\subseteq{\widetilde{C}}_{\alpha+A,A}^{\bullet+1}(\mathfrak{g}\oplus V,V).$$
 Define $d:{\widetilde{C}}_{\alpha+A,A}^p(\mathfrak{g}\oplus V,V)\longrightarrow{\widetilde{C}}_{\alpha+A,A}^{p+1}(\mathfrak{g}\oplus V,V)$  by
\begin{equation}
d(\varphi):=[\pi+\bar{\rho}+\bar{\nu},\varphi]^{3HL},\quad \varphi\in {\widetilde{C}}_{\alpha+A,A}^p(\mathfrak{g}\oplus V,V).
\end{equation}

\begin{thm}\label{thm:dd0}
Let $(V;\rho,\nu, A)$ be a generalized representation of a   $3$-Hom-Lie algebra $\mathfrak{g}$. Then $d\circ d =0.$ Thus, we obtain a new cohomology complex, where the space of $p$-Hom-cochains is given by ${\widetilde{C}}_{\alpha+A,A}^{p-1}(\mathfrak{g}\oplus V,V)$.
\end{thm}
\begin{proof} By the graded Jacobi identity, for any $\varphi\in{\widetilde{C}}_{\alpha+A,A}^{p-1}(\mathfrak{g}\oplus V,V)$, one obtains
$$
d\circ d(\varphi):=[\pi+\bar{\rho}+\bar{\nu},[\pi+\bar{\rho}+\bar{\nu},\varphi]^{3HL}]^{3HL}=\frac{1}{2} [[\pi+\bar{\rho}+\bar{\nu},\pi+\bar{\rho}+\bar{\nu}]^{3HL},\varphi]^{3HL}=0.
$$
\end{proof}

An element $\varphi\in {\widetilde{C}}_{\alpha+A,A}^{p-1}(\mathfrak{g}\oplus V,V)$ is called a $p$-cocycle if $d(\varphi)=0$; It is called a $p$-coboundary if there exists  $f \in {\widetilde{C}}_{\alpha+A,A}^{p-2}(\mathfrak{g}\oplus V,V)$ such that $\varphi=d(f)$.

Denote by $\mathcal{Z}_{3HL}^p(\mathfrak{g};V)$ and $\mathcal{B}_{3HL}^p(\mathfrak{g};V)$ the sets of $p$-cocycles and  $p$-coboundaries respectively. By Theorem \ref{thm:dd0}, we have $\mathcal{B}_{3HL}^p(\mathfrak{g};V)\subset\mathcal{Z}^p(\mathfrak{g};V)$. We define the $p$-th cohomolgy group
$\mathcal{H}_{3HL}^p(\mathfrak{g};V)$ to be $\mathcal{Z}_{3HL}^p(\mathfrak{g};V)/\mathcal{B}_{3HL}^p(\mathfrak{g};V)$.\\

The following  proposition provides a relationship  between this new cohomology and the one given by \eqref{eq:cohomology}.

\begin{proposition}
  There is  a forgetful map from $\mathcal{H}_{3HL}^p(\mathfrak{g};V)$ to $H_{3HL}^p(\mathfrak{g};V)$.
\end{proposition}
\begin{proof}
It is obvious that ${C}_{\alpha,A}^p(\mathfrak{g},V)\subseteq{\widetilde{C}}_{\alpha+A,A}^{p}(\mathfrak{g}\oplus V,V)$. By direct calculation, for $X_i\in\wedge^2 \mathfrak{g},z\in\mathfrak{g}$, we have
$$d(\varphi)(X_1,\cdots ,X_{p+1},z)=\delta_\rho(\varphi)(X_1,\cdots ,X_{p+1},z),\quad \varphi\in {C}_{\alpha,A}^p(\mathfrak{g},V),$$
where $\delta_\rho$ is the coboundary operator given by \eqref{eq:drho}. Thus, the natural projection from ${\widetilde{C}}_{\alpha+A,A}^{p}(\mathfrak{g}\oplus V,V)$ to ${C}_{\alpha,A}^p(\mathfrak{g},V)$ induces a forgetful map from $\mathcal{H}_{3HL}^p(\mathfrak{g};V)$ to $H_{3HL}^p(\mathfrak{g};V)$. \end{proof}

In the sequel, we give some characterization of  low dimensional cocycles.

\begin{proposition}
A linear map $\varphi\in Hom(\mathfrak{g},V)$ is a $1$-cocycle if only if for all $x_1,x_2,x_3\in \mathfrak{g},v\in V$, the following identities hold :
\begin{eqnarray*}
& \varphi\circ \alpha = A \circ \varphi,\\
  & \nu(x_1)(\varphi(x_2),v)-\nu(x_2)(\varphi(x_1),v)=0,\\
  & \varphi([x_1,x_2,x_3])-\rho(\alpha(x_1),\alpha(x_2))(\varphi(x_3))-\rho(\alpha(x_2),\alpha(x_3))(\varphi(x_1))-\rho(\alpha(x_3),\alpha(x_1))(\varphi(x_2))=0.
\end{eqnarray*}
\end{proposition}
\begin{proof}
For $\varphi\in Hom(\mathfrak{g},V)$ satisfying $\varphi\circ \alpha = A \circ \varphi$,  we have
\begin{eqnarray*}
d(\varphi)(x_1,x_2,v)=\nu(x_1)(\varphi(x_2),v)-\nu(x_2)(\varphi(x_1),v),
\end{eqnarray*}
and
\begin{eqnarray*}
&& d(\varphi)(x_1,x_2,x_3)=\delta_\rho(\varphi)(x_1,x_2,x_3)\\
&&=\rho(\alpha(x_1),\alpha(x_2))(\varphi(x_3))+\rho(\alpha(x_2),\alpha(x_3))(\varphi(x_1))+\rho(\alpha(x_3),\alpha(x_1))(\varphi(x_2))-\varphi([x_1,x_2,x_3]). \end{eqnarray*}
 \end{proof}
\begin{proposition}

  A $2$-cochain $\varphi_1+\varphi_2+\varphi_3\in {\widetilde{C}}_{\alpha,A}^{1}(\mathfrak{g}\oplus V,V)$, where $ \varphi_1\in Hom(\wedge^2 V\wedge \mathfrak{g},V),~\varphi_2\in Hom(\wedge^2 \mathfrak{g}\wedge V,V),~\varphi_3\in Hom(\wedge^3 \mathfrak{g},V)$, is a $2$-cocycle if and only if for all $x_i\in
 \mathfrak{g},v_j\in V$ and $v\in V$, the following identities hold:
 {\small\begin{align}
  \nonumber 0&=-\rho(\alpha(x_1),\alpha(x_2))(\varphi_3(x_3,x_4,x_5))-\varphi_3(\alpha(x_1),\alpha(x_2),[x_3,x_4,x_5])+\rho(\alpha(x_4),\alpha(x_5))(\varphi_3(x_1,x_2,x_3))
 \\& \nonumber+\varphi_3([x_1,x_2,x_3],\alpha(x_4),\alpha(x_5))+\rho(\alpha(x_5),\alpha(x_3))(\varphi_3(x_1,x_2,x_4))+\varphi_3(\alpha(x_3),[x_1,x_2,x_4],\alpha(x_5))
 \\\label{eq:2cocycle1}&+\rho(\alpha(x_3),\alpha(x_4))(\varphi_3(x_1,x_2,x_5))+\varphi_3(\alpha(x_3),\alpha(x_4),[x_1,x_2,x_5]),\\
\nonumber0&=\nu(\alpha(x_4))(\varphi_3(x_1,x_2,x_3),A(v))+\nu(\alpha(x_3))(A(v),\varphi_3(x_1,x_2,x_4))+\rho(\alpha(x_1),\alpha(x_2))(\varphi_2(x_3,x_4,v))
\\&\label{eq:2cocycle2}-\rho(\alpha(x_3),\alpha(x_4))(\varphi_2(x_1,x_2,v))-\varphi_2([x_1,x_2,x_3],\alpha(x_4),A(v))-\varphi_2(\alpha(x_3),[x_1,x_2,x_4],A(v)),
\\
\nonumber0&=\nu(\alpha(x_1))(A(v),\varphi_3(x_2,x_3,x_4))+\rho(\alpha(x_3),\alpha(x_4))(\varphi_2(x_1,x_2,v))-\rho(\alpha(x_2),\alpha(x_4))(\varphi_2(x_1,x_3,v))
\\&\nonumber+\rho(\alpha(x_2),\alpha(x_3))(\varphi_2(x_1,x_4,v))+\varphi_2(\alpha(x_3),\alpha(x_4),\rho(x_1,x_2)(v))-\varphi_2(\alpha(x_2),\alpha(x_4),\rho(x_1,x_3)(v))
\\&\label{eq:2cocycle3}+\varphi_2(\alpha(x_2),\alpha(x_3),\rho(x_1,x_4)(v))-\varphi_2(\alpha(x_1),[x_2,x_3,x_4],A(v)),\\
\nonumber&0=\nu(\alpha(x_3))(A(v_2),\varphi_2(x_1,x_2,v_1))+\nu(\alpha(x_3))(\varphi_2(x_1,x_2,A(v_2)),v_1)+\varphi_2(\alpha(x_1),\alpha(x_2),\nu(x_3)(v_1,v_2))
\\&\nonumber+\rho(\alpha(x_1),\alpha(x_2))(\varphi_1(v_1,v_2,x_3))-\varphi_1(\rho(x_1,x_2)(v_1),A(v_2),\alpha(x_3))-\varphi_1(A(v_1),\rho(x_1,x_2)(v_2),\alpha(x_3))
\\&\label{eq:2cocycle4}-\varphi_1(A(v_1),A(v_2),[x_1,x_2,x_3]),\\
\nonumber0&=\nu(\alpha(x_3))(A(v_2),\varphi_2(x_2,x_1,v_1))+\nu(\alpha(x_2))(\varphi_2(x_3,x_1,v_1),A(v_2))-\nu(\alpha(x_1))(A(v_1),\varphi_2(x_2,x_3,v_2))
\nonumber\\&\nonumber+\varphi_2(\alpha(x_2),\alpha(x_3),\nu(x_1)(v_1,v_2))+\rho(\alpha(x_2),\alpha(x_3))(\varphi_1(v_1,v_2,x_1))+\varphi_1(\rho(x_1,x_2)(v_1),A(v_2),\alpha(x_3))
\\&\label{eq:2cocycle5}-\varphi_1(A(v_1),\rho(x_2,x_3)(v_2),\alpha(x_1))+\varphi_1(A(v_2),\rho(x_1,x_3)(v_1),\alpha(x_2)),\\
\nonumber0&=\varphi_2(\alpha(x_1),\alpha(x_3),\nu(x_2)(v_1,v_2))-\varphi_2(\alpha(x_2),\alpha(x_3),\nu(x_1)(v_1,v_2))-\varphi_2(\alpha(x_1),\alpha(x_2),\nu(x_3)(v_1,v_2))
\\&\nonumber+\rho(\alpha(x_1),(x_3))(\varphi_1(v_1,v_2,x_2))-\rho(\alpha(x_1),\alpha(x_2))(\varphi_1(v_1,v_2,x_3))-\rho(\alpha(x_2),\alpha(x_3))(\varphi_1(v_1,v_2,x_1))
\\&\label{eq:2cocycle6}+\varphi_1(A(v_1),A(v_2),[x_1,x_2,x_3]),\\
\nonumber0&=-\nu(\alpha(x_2))(\varphi_1(v_1,v_2,x_1),A(v_3))-\nu(\alpha(x_2))( A(v_2),\varphi_1(v_1,v_3,x_1))+\nu(\alpha(x_1))( A(v_1),\varphi_1(v_2,v_3,x_2))
\\&\label{eq:2cocycle7}-\varphi_1(\nu(x_1)(v_1,v_2),A(v_3),\alpha(x_2))-\varphi_1(A(v_2),\nu(x_1)(v_1,v_3),\alpha(x_2))+\varphi_1(A(v_1),\nu(x_2)(v_2,v_3),\alpha(x_1)),\\
\nonumber0&=\nu(\alpha(x_2))(\varphi_1(v_1,v_2,x_1), A(v_3))-\nu(\alpha(x_1))( A(v_3),\varphi_1(v_1,v_2,x_2))+\varphi_1(\nu(x_1)(v_1,v_2),A(v_3),\alpha(x_2))\\
&-\varphi_1(\nu(x_2)(v_1,v_2),A(v_3),\alpha(x_1))\label{eq:2cocycle8}.
  \end{align}}
\end{proposition}
\begin{proof}
For $\varphi_3\in Hom(\wedge^3 \mathfrak{g},V)$, we have
\begin{eqnarray*}
d(\varphi_3)(x_1,x_2,x_3,x_4,x_5)&=&\rho(\alpha(x_1),\alpha(x_2))(\varphi_3(x_3,x_4,x_5))+\varphi_3(\alpha(x_1),\alpha(x_2),[x_3,x_4,x_5])\\&& \nonumber-\rho(\alpha(x_4),\alpha(x_5))(\varphi_3(x_1,x_2,x_3))
 -\varphi_3([x_1,x_2,x_3],\alpha(x_4),\alpha(x_5))\\&& \nonumber-\rho(\alpha(x_5),\alpha(x_3))(\varphi_3(x_1,x_2,x_4))-\varphi_3(\alpha(x_3),[x_1,x_2,x_4],\alpha(x_5))
 \\&& \nonumber-\rho(\alpha(x_3),\alpha(x_4))(\varphi_3(x_1,x_2,x_5))-\varphi_3(\alpha(x_3),\alpha(x_4),[x_1,x_2,x_5]),\\
d(\varphi_3)(x_1,x_2,x_3,x_4,v)&=&\nu(\alpha(x_4))(\varphi_3(x_1,x_2,x_3),A(v))+\nu(\alpha(x_3))(A(v),\varphi_3(x_1,x_2,x_4)),\\
d(\varphi_3)(x_1,v,x_2,x_3,x_4)&=&\nu(\alpha(x_1))(A(v),\varphi_3(x_2,x_3,x_4)).
\end{eqnarray*}
For $\varphi_2\in Hom(\wedge^2 \mathfrak{g}\wedge V,V)$, we have

\begin{eqnarray*}
d(\varphi_2)(x_1,x_2,x_3,x_4,v)&=&\rho(\alpha(x_1),\alpha(x_2))(\varphi_2(x_3,x_4,v))-\rho(\alpha(x_3),\alpha(x_4))(\varphi_2(x_1,x_2,v))\\
&&-\varphi_2([x_1,x_2,x_3],\alpha(x_4),A(v))-\varphi_2(\alpha(x_3),[x_1,x_2,x_4],A(v)),\\
d(\varphi_2)(x_1,v,x_2,x_3,x_4)&=&\rho(\alpha(x_3),\alpha(x_4))(\varphi_2(x_1,x_2,v))-\rho(\alpha(x_2),\alpha(x_4))(\varphi_2(x_1,x_3,v))\\
&&+\rho(\alpha(x_2),\alpha(x_3))(\varphi_2(x_1,x_4,v))+\varphi_2(\alpha(x_3),\alpha(x_4),\rho(x_1,x_2)(v))\\
&&-\varphi_2(\alpha(x_2),\alpha(x_4),\rho(x_1,x_3)(v))+\varphi_2(\alpha(x_2),\alpha(x_3),\rho(x_1,x_4)(v))\\
&&-\alpha_2(x_1,[x_2,x_3,x_4],v),
\end{eqnarray*}
\begin{eqnarray*}
d(\varphi_2)(x_1,x_2,v_1,v_2,x_3)&=&\nu(\alpha(x_3))(A(v_2),\varphi_2(x_1,x_2,v_1))+\nu(\alpha(x_3))(\varphi_2(x_1,x_2,v_2),A(v_1))\\
&&+\varphi_2(\alpha(x_1),\alpha(x_2),\nu(x_3)(v_1,v_2)),\\
d(\varphi_2)(x_1,v_1,x_2,v_2,x_3)&=&\nu(\alpha(x_3))(A(v_2),\varphi_2(x_2,x_1,v_1))+\nu(\alpha(x_2))(\varphi_2(x_3,x_1,v_1),A(v_2))\\
&&-\nu(\alpha(x_1))(A(v_1),\alpha_2(x_2,x_3,v_2))+\alpha_2(\alpha(x_2),\alpha(x_3),\nu(x_1)(v_1,v_2)),\\
d(\varphi_2)(v_1,v_2 ,x_1,x_2,x_3)&=&\varphi_2(\alpha(x_1),\alpha(x_3),\nu(x_2)(v_1,v_2))-\varphi_2(\alpha(x_2),\alpha(x_3),\nu(x_1)(v_1,v_2))\\&&-\varphi_2(\alpha(x_1),\alpha(x_2),\nu(x_3)(v_1,v_2)).
\end{eqnarray*}
For $\varphi_1\in Hom(\wedge^2 V\wedge \mathfrak{g},V)$, we have
\begin{eqnarray*}
d(\varphi_1)(x_1,x_2,v_1,v_2,x_3)&=&\rho(\alpha(x_1),\alpha(x_2))(\varphi_1(v_1,v_2,x_3))-\varphi_1(\rho(x_1,x_2)(v_1),A(v_2),\alpha(x_3))\\
&&-\varphi_1(A(v_1),\rho(x_1,x_2)(v_2),\alpha(x_3))-\varphi_1(A(v_1),A(v_2),[x_1,x_2,x_3]),\\
d(\varphi_1)(x_1,v_1,x_2,v_2,x_3)&=&\rho(\alpha(x_2),\alpha(x_3))(\varphi_1(v_1,v_2,x_1))+\varphi_1(\rho(x_1,x_2)(v_1),A(v_2),\alpha(x_3))\\
&&-\varphi_1(A(v_1),\rho(x_2,x_3)(v_2),\alpha(x_1))+\varphi_1(A(v_2),\rho(x_1,x_3)(v_1),\alpha(x_2)),\\
d(\varphi_1)(v_1,v_2 ,x_1,x_2,x_3)&=&\rho(\alpha(x_1),\alpha(x_3))(\varphi_1(v_1,v_2,x_2))-\rho(\alpha(x_1),\alpha(x_2))(\varphi_1(v_1,v_2,x_3))\\
&&-\rho(\alpha(x_2),\alpha(x_3))(\varphi_1(v_1,v_2,x_1))+\varphi_1(A(v_1),A(v_2),[x_1,x_2,x_3]),\\
d(\varphi_1)(x_1,v_1,v_2,v_3,x_2)&=&-\nu(\alpha(x_2))(\varphi_1(v_1,v_2,x_1),A(v_3))-\nu(\alpha(x_2))( A(v_2),\varphi_1(v_1,v_3,x_1))\\&&+\nu(\alpha(x_1))( A(v_1),\varphi_1(v_2,v_3,x_2))
-\varphi_1(\nu(x_1)(v_1,v_2),A(v_3),\alpha(x_2))\\&&-\varphi_1(A(v_2),\nu(x_1)(v_1,v_3),\alpha(x_2))+\varphi_1(A(v_1),\nu(x_2)(v_2,v_3),\alpha(x_1)),\\
d(\varphi_1)(v_1,v_2 ,x_1,x_2,v_3)&=&\nu(\alpha(x_2))(\varphi_1(v_1,v_2,x_1),A( v_3))-\nu(\alpha(x_1))(A( v_3),\varphi_1(v_1,v_2,x_2))\\
&&+\varphi_1(\nu(x_1)(v_1,v_2),A(v_3),\alpha(x_2))-\varphi_1(\nu(x_2)(v_1,v_2),A(v_3),\alpha(x_1)).
\end{eqnarray*}
Thus, $d(\varphi_1+\varphi_2+\varphi_3)=0$ if and only if Eqs \eqref{eq:2cocycle1}-\eqref{eq:2cocycle8} hold. \end{proof}

In the following we provide two examples of computation of $2$-cocycles of the $3$-dimensional   $3$-Hom-Lie algebra.

\begin{example}{\rm
 Let $(\mathfrak{g}, [\cdot,\cdot,\cdot ],\alpha)$ be the $3$-dimensional   $3$-Hom-Lie algebra defined, with respect to a  basis $\{e_1,e_2,e_3\}$, by  $[e_1, e_2, e_3]=a_1 e_1, ~\alpha(e_1)=a_1 e_1, ~\alpha(e_2)= a_2 e_2, ~ \alpha(e_3)=\frac{1}{a_2}  e_3$.
 Let $V$ be a 2-dimensional vector space,  $\{v_1,v_2\}$ its basis and $A\in gl(V)$ defined by: $A(v_1)=a_1 v_1, A(v_2)=\frac{a_2 a_3}{a_1} v_2,$ where $a_1,a_2,a_3$ are parameters in $\K$.

We consider the generalized representation $(V,\rho,\nu, A)$, where   $\rho$ and $\nu$ are defined with respect to the basis by
\begin{align*}
   & \rho(e_1,e_2)(v_1)=0,  \;
\rho(e_1,e_2)(v_2)=0 ,   \;
 \rho(e_1,e_3)(v_1)=0, \\ &
  \rho(e_1,e_3)(v_2)=r_2 a_1 v_1, \;
 \rho(e_2,e_3)(v_1)=a_1 v_1,    \;  \rho(e_2,e_3)(v_2)=\frac{r_1 a_2 a_3}{a_1}v_2 ,
 \\ &
 \nu(e_1)(v_1,v_2)=0,  \;
\nu(e_2)(v_1,v_2)=0,  \;
\nu(e_3)(v_1,v_2)=s_1 a_1 v_1,
\end{align*}
with $s, r_1, r_2$   parameters in $\K$ and $a_1a_2s\neq0$.

We have the following   $2$-cocycles : $\varphi_1=0$, $\varphi_3=0$ and $\varphi_2$  defined as
 \begin{eqnarray*}
& \varphi_2 (e_1,e_2,v_1)=0,\
 & \varphi_2 (e_1,e_2,v_2)=0,\\
& \varphi_2 (e_1,e_3,v_1)=c_1 v_1,\
 & \varphi_2 (e_1,e_3,v_2)=0,\\
& \varphi_2 (e_2,e_3,v_1)=c_2 v_1,\
 & \varphi_2 (e_2,e_3,v_2)=c_3 v_2,
\end{eqnarray*}
where $c_1,c_2, c_3$ are parameters in $\K$.
}
\end{example}

 \section{Abelian extensions of    $3$-Hom-Lie algebras }
 In this section, we show that associated to any abelian extension, there is a
generalized representation and a $2$-cocycle.
\begin{df}
  Let $(\mathfrak{g}, [\cdot, \cdot,\cdot]_\mathfrak{g},\alpha)$,~ $(V, [\cdot,\cdot,\cdot]_V ,A )$,~ and $(\mathfrak{\hat{g}}, [\cdot,\cdot,\cdot ]_{\mathfrak{\hat{g}}}, \alpha_{\mathfrak{\hat{g}}} )$ be   $3$-Hom-Lie algebras and $i : V \longrightarrow \mathfrak{\hat{g}}$, $ p :\mathfrak{\hat{g}} \longrightarrow \mathfrak{g}$ be morphisms of   $3$-Hom-Lie algebras.
   The following sequence of    $3$-Hom-Lie algebras is a short exact sequence if $Im(i) =Ker(p)$, $ Ker(i) = 0 $ and $Im(p) = \mathfrak{g}$:
$$0\longrightarrow V\stackrel{i}{\longrightarrow} \mathfrak{\hat{g}}\stackrel{p}{\longrightarrow} \mathfrak{g}\longrightarrow 0$$
where $A(V) = \alpha_{\mathfrak{\hat{g}}}(V )$.
In this case, we call ~$\mathfrak{\hat{g}}$ an extension of $\mathfrak{g}$ by $V$, and denote it by $E_{\mathfrak{\hat{g}}}$ . It is called
an abelian extension if $V$ is an abelian ideal of $\mathfrak{\hat{g}}$, i.e., $[u, v]_{V} = 0$ for all $u, v \in V $.
A section $\sigma$ of $p : \mathfrak{\hat{g}} \longrightarrow \mathfrak{g}$ consists of linear maps $ \sigma : \mathfrak{g\longrightarrow} \mathfrak{\hat{g}}$ such that $p \circ \sigma= id_{\mathfrak{g}}$
and $\sigma\circ \alpha =\alpha_{\mathfrak{\hat{g}}}\circ\sigma.$
 \end{df}
 \begin{df}
 Two extensions of   $3$-Hom-Lie  algebras\\  $
\xymatrix@C=0.5cm{
 E_{\hat{\mathfrak{g}}}: 0 \ar[r] &V \ar[rr]^{i_1} && \hat{\mathfrak{g}} \ar[rr]^{p_1} && \mathfrak{g} \ar[r] & 0 },
  $  and  $
\xymatrix@C=0.5cm{
 E_{\tilde{\mathfrak{g}} }: 0 \ar[r] &V \ar[rr]^{i_2} && \tilde{\mathfrak{g}} \ar[rr]^{p_2} && \mathfrak{g} \ar[r] & 0 },
  $ \\ are equivalent if there exists a morphism of   $3$-Hom-Lie  algebras $\phi:\mathfrak{\hat{g}}\longrightarrow \mathfrak{\tilde{g}}$ such that the following diagram commutes:

\newcommand{\Lrightarrow}{\hbox to1cm{\rightarrowfill}}
\newcommand{\Ldownarrow}{\bigg\downarrow}

\[
  \setlength{\arraycolsep}{1pt}
  \begin{array}{*{9}c}
    0 &\Lrightarrow & V & \stackrel{i_1}{\Lrightarrow} & \hat{\mathfrak{g}} & \stackrel{p_1}{\Lrightarrow} & \mathfrak{g} & \Lrightarrow & 0\\
    & & \Ldownarrow\mbox{Id}_V & & \Ldownarrow \phi & & \Ldownarrow \mbox{Id}_{\mathfrak{g}}& & \\
 0 &\Lrightarrow & V & \stackrel{i_2}{\Lrightarrow} & \widetilde{\mathfrak{g}} & \stackrel{p_2}{\Lrightarrow} & \mathfrak{g} & \Lrightarrow & 0
  \end{array}
\]

 \end{df}
 A linear map
$\sigma : \mathfrak{g}\longrightarrow\mathfrak{\hat{g}}$ is called a splitting of $\mathfrak{g}$ if it satisfies $p\circ\sigma= id_{\mathfrak{g}}$. If there exists a splitting which is also a homomorphism between   $3$-Hom-Lie algebras, we say that the abelian extension is split.
Let $\hat{\mathfrak{g}}$ be a split abelian extension and $\sigma:\mathfrak{g}\longrightarrow \hat{\mathfrak{g}}$ the corresponding splitting. Define $\rho :\wedge^2 \mathfrak{g}\longrightarrow \mathfrak{gl}(V)$ and $\nu: \mathfrak{g}\longrightarrow Hom (\wedge^2 V, V)$ by
\begin{align*}
\rho(x,y)(u)=&[\sigma(x),\sigma(y),u]_{\hat{\mathfrak{g}}},\\
\nu(x)(u,v)=&[\sigma(x),u,v]_{\hat{\mathfrak{g}}}.
\end{align*}
Then, we can transfer the    $3$-Hom-Lie algebra structure on $\hat{\mathfrak{g}}$ to that on  $\mathfrak{g}\oplus V$  in terms of $\rho$ and $\nu$:

Note that the Hom-Filippov-Jacobi identity gives the character of  $\rho$ and $\nu$.
\begin{equation*}[x + u, y + v, z + w]_{(\rho,\nu)} = [x, y, z] + \rho(x, y)(w) + \rho(y, z)(u) + \rho(z, x)(v)
+\nu(x)(v \wedge w) + \nu(y)(w \wedge u) + \nu(z)(u \wedge v).\end{equation*}
 However, by Theorem \ref{Theoremde repgeneraliseAvecPRODUITSEMIdirecteGNERALISE}, it
is straightforward to obtain the following proposition.
\begin{proposition}
  Any split abelian extension of   $3$-Hom-Lie algebras is isomorphic to a generalized semidirect of product   $3$-Hom-Lie algebra.
\end{proposition}

 Now, for non-split abelian extensions, we can further define $\omega:\wedge^3\mathfrak{g}\longrightarrow V$ by
\begin{eqnarray*}
  \omega(x,y,z)&=&[\sigma(x),\sigma(y),\sigma(z)]_{\hat{\mathfrak{g}}}-\sigma[x,y,z]_{\mathfrak{g}}.
\end{eqnarray*}
Then, we also transfer the   $3$-Hom-Lie algebra  structure on $\hat{\mathfrak{g}}$ to that on $\mathfrak{g}\oplus V$ in terms of $\rho,\nu$ and $\omega:$
\begin{eqnarray*}
[x_1+v_1,x_2+v_2,x_3+v_3]_{(\rho,\nu,\omega)}&=&[x_1,x_2,x_3]_{\mathfrak{g}}+\rho(x_1,x_2)(v_3)+\rho(x_3,x_1)(v_2)+\rho(x_2,x_3)(v_1)\\
&&+\nu(x_1)(v_2,v_3)+\nu(x_2)(v_3,v_1)+\nu(x_3)(v_1,v_2)+\omega(x_1,x_2,x_3).
\end{eqnarray*}
\begin{thm}\label{thm:abelian ext}
 With above notations, $(\mathfrak{g}\oplus V,[\cdot,\cdot,\cdot]_{(\rho,\nu,\omega)},\alpha_{\mathfrak{g}\oplus V})$ is a $3$-Hom-Lie algebra if and only if for all $x_1,x_2,x_3,x_4,x_5\in \mathfrak{g}$ and $v,v_1,v_2,v_3\in V$, Eqs. \eqref{eq3deRepresenGenerelise}-\eqref{eq6deRepresenGenerelise} and the following identities hold:
    \begin{eqnarray}
 \nonumber 0&=&-\rho(\alpha(x_1),\alpha(x_2))(\omega(x_3,x_4,x_5))-\omega(\alpha(x_1),\alpha(x_2),[x_3,x_4,x_5])+\rho(\alpha(x_4),\alpha(x_5))(\omega(x_1,x_2,x_3))
 \\&& \nonumber+\omega([x_1,x_2,x_3],\alpha(x_4),\alpha(x_5))+\rho(\alpha(x_5),\alpha(x_3))(\omega(x_1,x_2,x_4))+\omega(\alpha(x_3),[x_1,x_2,x_4],\alpha(x_5))
 \\\label{eq:t1}&&+\rho(\alpha(x_3),\alpha(x_4))(\omega(x_1,x_2,x_5))+\omega(\alpha(x_3),\alpha(x_4),[x_1,x_2,x_5]),\\
 \nonumber 0&=&\nu(\alpha(x_1))(A(v),\omega(x_2,x_3,x_4))+\rho([x_2,x_3,x_4],\alpha(x_1))(A(v))+\rho(\alpha(x_3),\alpha(x_4))\rho(x_1,x_2)(v)
\\\label{eq:t2} &&-\rho(\alpha(x_2),\alpha(x_4))\rho(x_1,x_3)(v)
+\rho(\alpha(x_2),\alpha(x_3))\rho(x_1,x_4)(v),\\
 \nonumber0&=&\rho(\alpha(x_1),\alpha(x_2))\rho(x_3,x_4)(v) - \rho(\alpha(x_3),\alpha(x_4))\rho(x_1,x_2)(v) -\rho([x_1,x_2,x_3],\alpha(x_4))(A(v))
 \label{eq:t3} \\&&-   \nu(\alpha(x_4))(A(v),\omega(x_1,x_2,x_3))-\rho(\alpha(x_3),[x_1,x_2,x_4])(A(v))+ \nu(\alpha(x_3))(A(v),\omega(x_1,x_2,x_4)).
  \end{eqnarray}

\end{thm}
The Fundamental Identity  gives the character of  $\rho$, $\nu$ and $\omega$.

\begin{proof} The pair  $(\mathfrak{g}\oplus V,[\cdot,\cdot,\cdot]_{(\rho,\nu,\omega)},\alpha_{\mathfrak{g}\oplus V})$ defines a $3$-Hom-Lie algebra if and only if  the Hom-Filippov-Jacobi identity holds on all elements of $ \mathfrak{g}\oplus V$. Condition \eqref{eq:t1} is obtained using the Hom-Filippov-Jacobi identity on $\{x_1,x_2,x_3,x_4,x_5\}$ elements of $\mathfrak{g}$.


Similarly, elements $\{x_1,v,x_2,x_3,x_4\}$ gives Eq. \eqref{eq:t2},  $\{x_1,x_2,v,x_3,x_4\}$ gives Eq. \eqref{eq:t3}, $\{x_1,x_2, v_1,v_2, x_3 \}$ gives Eq. \eqref{eq3deRepresenGenerelise},  $\{v_1,x_1, v_2,x_2, x_3 \}$ gives Eq. \eqref{eq4deRepresenGenerelise},  
$\{v_1,x_1, v_2,v_3,x_2 \}$ gives Eq. \eqref{eq5deRepresenGenerelise} and  $\{v_1,v_2,v_3, x_1,x_2 \}$ gives Eq. \eqref{eq6deRepresenGenerelise}.

Conversely, if Eqs. \eqref{eq3deRepresenGenerelise}-\eqref{eq6deRepresenGenerelise} and Eqs. \eqref{eq:t1}-\eqref{eq:t3} hold, it is straightforward to see that for all $e_1,\cdots,e_5\in \mathfrak{g}\oplus V$, the  Hom-Filippov-Jacobi identity holds. Thus, $(\mathfrak{g}\oplus V,[\cdot,\cdot,\cdot]_{(\rho,\nu,\omega)},\alpha_{\mathfrak{g}\oplus V})$ is a   $3$-Hom-Lie algebra.  \end{proof}

\textbf{Acknowledgment} : The authors would like to thank Yunhe Sheng for his comments and suggestions.

\end{document}